\documentclass[smallextended,referee,envcountsect,]{svjour3}

\smartqed
\usepackage{graphicx}

\usepackage{amsmath,amssymb,amsopn,amsfonts, mathabx}
\usepackage{enumitem}
\usepackage{float}
\usepackage{bm}
\usepackage{tikz}
\usepackage{tikz-3dplot}
\usepackage{hyperref}

\usetikzlibrary{arrows.meta}
\usetikzlibrary{patterns}

\usepackage{caption,subcaption}
\usepackage{newfloat}

\usepackage{amsopn}

\DeclareMathOperator{\dom}{dom}
\DeclareMathOperator{\range}{range}

\DeclareMathOperator{\codim}{codim}
\DeclareMathOperator{\col}{col}
\DeclareMathOperator{\row}{row}
\DeclareMathOperator{\epi}{epi}
\DeclareMathOperator{\graph}{graph}

\DeclareMathOperator{\relint}{relint}

\DeclareMathOperator{\aff}{aff}
\DeclareMathOperator*{\argmin}{argmin}
\DeclareMathOperator*{\argmax}{argmax}
\DeclareMathOperator{\Rank}{rank}
\DeclareMathOperator{\Nullity}{nullity}

\DeclareMathOperator{\conv}{conv}

\DeclareMathOperator{\nullspace}{Null}

\newcommand{\reals}{\mathbb{R}}

\newcommand{\complex}[1]{\mathcal{#1}}
\newcommand{\ntext}[1]{\,\text{\normalfont #1}}

\renewcommand{\vec}[1]{\bm{#1}}
\newcommand{\set}[1]{\mathbf{#1}}

\DeclareUnicodeCharacter{2113}{$l$}

\journalname{JOTA}

\begin{document}

\title{The Geometry and Well-Posedness of Sparse Regularized Linear Regression}

\subtitle{}

\author{Jasper M. Everink,Yiqiu Dong and Martin S. Andersen}

\institute{
Department of Applied Mathematics and Computer Science, Technical University of Denmark. Richard Petersens Plads, Building 324, DK-2800 Kgs. Lyngby, Denmark. (\href{mailto:jmev@dtu.dk}{jmev@dtu.dk}, \href{mailto:yido@dtu.dk}{yido@dtu.dk}, \href{mailto:mskan@dtu.dk}{mskan@dtu.dk})
}

\date{}

\maketitle

\begin{abstract}
In this work, we study the well-posedness of certain sparse regularized linear regression problems, i.e., the existence, uniqueness and continuity of the solution map with respect to the data. We focus on regularization functions that are convex piecewise linear, i.e., whose epigraph is polyhedral. This includes total variation on graphs and polyhedral constraints. We provide a geometric framework for these functions based on their connection to polyhedral sets and apply this to the study of the well-posedness of the corresponding sparse regularized linear regression problem. Particularly, we provide geometric conditions for well-posedness of the regression problem, compare these conditions to those for smooth regularization, and show the computational difficulty of verifying these conditions.
\end{abstract}

\section{Introduction}\label{sec:introduction}

A common goal within finite-dimensional linear inverse problems is to estimate parameters $\vec{x} \in \reals^n$ from measurements $\vec{b} \in \reals^m$. We assume  that the measurements $\vec{b}$ were obtained by observing $\vec{x}$ through some forward operator $A \in \reals^{m\times n}$ such that $\vec{b} = A\vec{x} + \vec{\epsilon}$, where $\vec{\epsilon}$ is the error in this approximation. The same has been studied in statistics, where the goal of linear regression is to find explanatory variables $\vec{x}$ that explain the measured variables $\vec{b}$ through a linear model $\vec{b} = A\vec{x} + \vec{\epsilon}$.

One approach to estimating the parameters $\vec{x}$ is by solving a regularized linear least squares problem of the form
\begin{equation}\label{eq:lls_introduction}
    \argmin_{\vec{x} \in \reals^n}\left\{\frac{1}{2}\|A\vec{x} - \vec{b}\|_2^2 + f(\vec{x})\right\},
\end{equation}
where the regularization $f:\reals^n \rightarrow \reals \cup \{\infty\}$ is chosen to penalize unwanted behavior. Common choices for the regularization functions are (generalized) Tikhonov regularization $\|L\vec{x} - \vec{c}\|_2^2$, sparsity promoting regularization like $\|L\vec{x}\|_1$ and $\sum_{i=1}^{k} \|L_i\vec{x}\|_2$, constraints $\chi_{\set{C}}(\vec{x})$, where $\chi_{\set{C}}(\vec{x})$ is zero for $\vec{x}$ in $\set{C}$ and infinite otherwise, and conical combinations of any of these.

Regularization serves partially to impose certain behavior on the solution, but also to guarantee that the problem has a solution and that this solution is unique. After all, a linear system $A\vec{x} = b$ need not have a solution or can have infinitely many solutions. We refer to a problem, e.g., an inverse problem or an optimization problem like \eqref{eq:lls_introduction}, as well-posed if the following postulates holds for all $\vec{b}$:
\begin{enumerate}
    \item There exists a solution,
    \item The solution is unique,
    \item The solution depends continuously on $\vec{b}$.
\end{enumerate}

The existence gives us at least one solution to reason about. The uniqueness guarantees that there do not exist multiple solutions that might behave differently from each other. Finally, the continuity of the solution guarantees that small perturbations of $\vec{b}$ do not dramatically alter the solution. If these conditions do not holds, we refer to a problem as being ill-posed.

Whilst inverse problems are often ill-posed, they are put into different frameworks, like the regularized linear least squares problem \eqref{eq:lls_introduction}, to make them well-posed. Verifying the well-posedness of \eqref{eq:lls_introduction} is relatively straightforward for some regularization functions, e.g., Tikhonov and affine constraints, but requires considerable more effort for others, e.g., sparsity promoting regularization. 

For many common regularization functions, existence of a solution to \eqref{eq:lls_introduction} can be easily verified using tools from convex analysis. However, certifying uniqueness is more difficult for sparsity promoting regularization. In \cite{tibshirani2013lasso,ali2019generalized}, they showed sufficient conditions for the uniqueness for regularization of the form $\|L\vec{x}\|_1$. Over the course of the three papers \cite{ewald2020distribution,schneider2022geometry,tardivel2021geometry}, necessary and sufficient conditions for uniqueness were proven if the regularization function is a support function of finite sets, i.e., $f(\vec{x}) = \sigma_{\set{V}}(\vec{x}) := \max_{\vec{v} \in \set{V}}\{\vec{v}^T\vec{x}\}$ where $\set{V} \subset \reals^n$ is finite. Examples of these functions are the generalized lasso $\|L\vec{x}\|_1$ and norms with a polyhedral unit ball, e.g.,  the Ordered Weighted $l_1$ (OWL) norm \cite{figueiredo2016ordered}.

Uniqueness of solutions for fixed data $\vec{b}$ and regularization with support functions has been studied in \cite{zhang2015necessary}. The support functions of finite sets are all convex and piecewise linear. The uniqueness for fixed data $\vec{b}$ of minimizing convex piecewise linear functions subject to affine constraints has been studied in \cite{gilbert2017solution} and for various other minimization problems in \cite{mousavi2019solution}. In \cite{everink2023sparse}, they studied the sparsity properties of the probability distribution obtained by randomizing the measurements $\vec{b}$ in \eqref{eq:lls_introduction} with convex piecewise linear regularization, assuming that the forward operator $A$ has full column rank. The uniqueness of nuclear norm minimization, for which the regularization is not piecewise linear everywhere, has been studied in \cite{hoheisel2023uniqueness}.

In this work, we generalize the necessary and sufficient conditions developed in \cite{ewald2020distribution,schneider2022geometry,tardivel2021geometry} to convex piecewise linear regularization functions. These are precisely the functions that behave affine almost everywhere. Particularly, we provide a geometric framework for studying convex piecewise linear functions and their sparsity properties when used for regularization. We show necessary and sufficient conditions for existence and uniqueness of the solutions of
\begin{equation}\label{eq:variational_introduction}
    \argmin_{\vec{x} \in \reals^n}\left\{ D(A\vec{x},\vec{b}) + f(\vec{x})\right\},
\end{equation}
where $D : \reals^m \times \reals^m \rightarrow \reals$ can be more general discrepancy function than the least squares error. For linear least squares problems regularized by convex piecewise linear functions, we furthermore show that uniqueness of the solution implies continuity, thereby providing necessary and sufficient conditions for well-posedness of this class of sparse regularized linear least squares problems. A simplified version of this result is formulated in the following theorem.
\begin{theorem}[Necessary and sufficient conditions for well-posedness]\label{thm:nec_suf_well_posedness_intro}
Let $A\in \reals^{m\times n}$ be a linear operator and $f:\reals^n \rightarrow \reals \cup \{\infty\}$ be a convex piecewise linear function, i.e., $f(\vec{x}) = \max_{i = 1, \dots k}\{\vec{v}_i^T\vec{x} + w_i\} + \chi_{\set{P}}(\vec{x})$ with $\vec{v}_i \in \reals^n$, $w_i \in \reals$, $\set{P}$ a polyhedral set and $\chi_{\set{P}}(\vec{x})$ is zero if $\vec{x} \in \set{P}$ and infinite otherwise. Furthermore, assume that $\dim(\dom(f)) = n$, where $\dom(f)$ is the effective domain of $f$, i.e., the points $\vec{x}$ in the domain for which $f(\vec{x})$ is finite. Then, the sparse regularized linear least squares problem
\begin{equation}\label{eq:linear_least_squares}
    \argmin_{\vec{x} \in \reals^n}\left\{\frac{1}{2}\|A\vec{x} - \vec{b}\|_2^2 + f(\vec{x})\right\},
\end{equation}
is well-posed if and only if $\row(A) \cap \partial f(\vec{x}) \neq \emptyset$ for $\vec{x} \in \reals^n$ implies that $\dim(\partial f(\vec{x})) \geq \Nullity(A)$, where $\row(A)$ is the row space of $A$, $\Nullity(A)$ is the dimension of the null space of $A$ and $\partial f(\vec{x})$ is the subdifferential of $f$ at $\vec{x}$.
\end{theorem}

Through such conditions, we discuss generalized LASSO and generalized Tikhonov regularization, i.e., penalizing with either $f(\vec{x}) = \|L\vec{x}\|_1$ or $f(\vec{x}) = \|L\vec{x}\|_2^2$. Furthermore, we discuss the combinatorial nature of the well-posedness when $A$ is almost full-rank, i.e., $\Nullity(A) = 1$, and use this to show the computational difficulty of showing well-posedness. Specifically, we show that for various sparsity promoting penalties, verifying well-posedness is a co-NP-hard problem.

This paper is organized as follows. In Section \ref{sec:convex_piecewise_linear_functions}, we discuss a geometric framework for studying convex piecewise linear functions and show that uniqueness of sparse regularized linear regression is mostly a geometry property of these functions. In Section \ref{sec:well-posedness}, we discuss the various postulates of well-posedness in relation to the geometric framework of the previous section and prove the necessary and sufficient conditions for uniqueness and well-posedness. Finally, in Section \ref{sec:implications}, we discuss various implications of the geometric conditions for uniqueness and well-posedness, including a comparison to the conditions for well-posedness with smooth Tikhonov regularization and the computational difficulty of verifying well-posedness.

\section{Convex piecewise linear functions}\label{sec:convex_piecewise_linear_functions}
From Fermat's rule, it follows that a vector $\vec{x}^\star \in \reals^n$ is a solution to the linear least squares problem \eqref{eq:linear_least_squares} if and only if
\begin{equation*}
    A^T(\vec{b} - A\vec{x}^\star) \in \partial f(\vec{x}^\star),
\end{equation*}
where $\partial f(\vec{x}^\star) := \{\vec{v} \in \reals^n\,|\, f(\vec{x}) \geq f(\vec{x}^\star) + \vec{v}^T(\vec{x} - \vec{x}^\star) \text{ for all } \vec{x} \in \reals^n\}$ denotes the subdifferential. Note that, if another vector $\vec{y}^\star \in \reals^n$ satisfies $A\vec{x}^\star = A\vec{y}^\star$ and $\partial f(\vec{x}^\star) = \partial f(\vec{y}^\star)$, then the line segment from $\vec{x}^\star$ to $\vec{y}^\star$ are all solutions to \eqref{eq:linear_least_squares}. Thus, to study the uniqueness of solutions, it can be useful to study subsets of $\reals^n$ that have the same subdifferential $\partial f$ and how these sets relate to the null space of $A$. For convex piecewise linear functions, this relation completely describes the uniqueness of solutions.

In this section, we present a geometric framework for studying convex piecewise linear functions and their subdifferentials. Specifically, in Subsection \ref{subsec:polyhedral_complexes} we explain polyhedral complexes and how they can be used to describe the geometry of convex piecewise linear functions and their subdifferentials. In Subsection \ref{subsec:duality}, we show how convex conjugacy of convex piecewise linear functions result in duality between polyhedral complexes. In Subsection \ref{subsec:TV}, we study the polyhedral complexes associated with a commonly used regularization function, Total Variation regulation on graphs. Finally, in Subsection \ref{subsec:level_sets_linear_subspaces}, the level sets of convex piecewise linear functions and their interaction with linear subspaces are studied to obtain geometric properties needed to prove uniqueness of sparse regularized linear regression, discussed in Section \ref{sec:well-posedness}.

\subsection{Polyhedral complexes}\label{subsec:polyhedral_complexes}

Recall that a set $\set{P} \subset \reals^n$ is polyhedral if there exist $\vec{v}_i \in \reals^n$ and $w_i \in \reals$ for $i = 1,\dots,k$, such that $\set{P} = \{\vec{x} \in \reals^n\,|\, \vec{v}_i^T\vec{x} \leq w_i \text{ for all } i = 1,\dots,k\}$. Thus, polyhedral sets are the intersection of finitely many halfspaces. Convex piecewise linear functions, also referred to as polyhedral functions, can be considered as the function equivalent of polyhedral sets, according to the following equivalent definitions.

\begin{definition}[convex piecewise linear function]\label{def:piecewise_linear_function}
A function $f: \reals^n \rightarrow \reals \cup \{\infty\}$ is called convex piecewise linear if any of the following equivalent conditions hold:
\begin{itemize}
    \item the epigraph $\epi(f):= \{(\vec{x}, t) \in \reals^n \times \reals \,|\, f(\vec{x}) \leq t \}$ is a polyhedral set.
    \item the domain of $f$ is a polyhedral set and it is the maximum over finitely many affine functions, i.e., $f(\vec{x}) = \max_{i = 1, \dots k}\{\vec{v}_i^T\vec{x} + w_i\} + \chi_{\set{P}}(\vec{x})$ with $\vec{v}_i \in \reals^n$, $w_i \in \reals$, $\set{P}$ a polyhedral set.
    \item $f$ is convex and the domain of $f$ can be covered by finitely many polyhedral sets, such that the restriction of $f$ to any of these polyhedral sets is an affine function.
\end{itemize}
\end{definition}

Some examples of convex piecewise linear functions are based on $l_1$-norm regularization functions of the form $\vec{x} \mapsto \|L\vec{x}-\vec{c}\|_1$ and polyhedral constraints $\vec{x} \mapsto \chi_{\set{P}}(\vec{x})$. Furthermore, the finite sum of convex piecewise linear functions is again a convex piecewise linear function. However, not all commonly used sparsity imposing regularization functions are convex piecewise linear functions, e.g., group lasso and isotropic total variation are of the form $\vec{x} \mapsto \sum_{i = 1}^{k} \|L_i\vec{x}- \vec{c}_i\|_2$, which need not be piecewise linear. For a more detailed study of convex piecewise linear functions, see \cite{rockafellar2009variational,mousavi2019solution}.

Consider the convex piecewise linear function $f(\vec{x}) = \|\vec{x}\|_1$. Let 
\begin{equation*}
    \set{I}(a) = \begin{cases}
        \{-1\} \quad &\text{if } a < 0,\\
        [-1,1] \quad &\text{if } a = 0,\\
        \{1\} \quad &\text{if } a > 0,\\
    \end{cases}
\end{equation*}
then the subdifferential of $f$ can be written as the following Cartesian product.
\begin{equation*} 
    \partial f(\vec{x}) = \prod_{i=1}^{n} \set{I}(\vec{x}_i).
\end{equation*}
Fixing the signs of the coordinates results in sets of equal subdifferential. Furthermore, the closure of these sets are all polyhedral and shown in Figure \ref{fig:l1_primal}. Furthermore, the corresponding subdifferentials are also polyhedral, because they are the Cartesian product of intervals. These subdifferentials are shown in Figure \ref{fig:l1_dual}. In Figure \ref{fig:examples_l1} and any similar further figure, circles represent sets of dimension zero, lines represent sets of dimension one, and dotted ares represent sets of dimension two. 

\newcommand{\scalefac}{0.8}

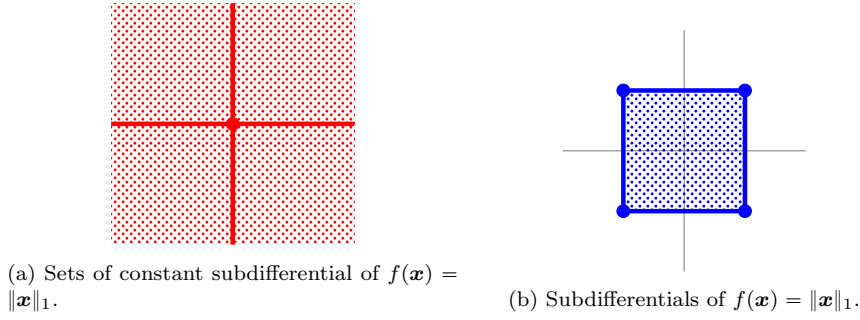
\begin{figure}
\centering
\begin{subfigure}{0.5\textwidth}
    \centering
    \begin{tikzpicture}[scale = \scalefac]
        \draw[gray] (-2,0) -- (2, 0);
        \draw[gray] (0,-2) -- (0, 2);
        
        \fill[pattern=crosshatch dots, pattern color=red] (0,0) rectangle ++(-2,-2);
        \fill[pattern=crosshatch dots, pattern color=red] (0,0) rectangle ++(2,2);
        
        \fill[pattern=crosshatch dots, pattern color=red] (0,0) rectangle ++(-2,2);
        \fill[pattern=crosshatch dots, pattern color=red] (0,0) rectangle ++(2,-2);
        
        \draw[red, line width = 0.2em] (0,0) -- (2, 0);
        \draw[red, line width = 0.2em] (0,0) -- (0, 2);
        \draw[red, line width = 0.2em] (0,0) -- (-2, 0);
        \draw[red, line width = 0.2em] (0,0) -- (0, -2);
        
        \filldraw[red] (0,0) circle (3pt);
        
    \end{tikzpicture}
    \caption{Sets of constant subdifferential of $f(\vec{x}) = \|\vec{x}\|_1$.}
    \label{fig:l1_primal}
\end{subfigure}%
\begin{subfigure}{0.5\textwidth}
\centering
    \begin{tikzpicture}[scale = \scalefac]
        \draw[gray] (-2,0) -- (2, 0);
        \draw[gray] (0,-2) -- (0, 2);
        
        \fill[pattern=crosshatch dots, pattern color=blue] (-1,-1) rectangle ++(2,2);
        
        \draw[blue, line width = 0.2em] (1,1) -- (1, -1) -- (-1, -1) -- (-1, 1) -- cycle;
        
        \filldraw[blue] (1,1) circle (3pt);
        \filldraw[blue] (1,-1) circle (3pt);
        \filldraw[blue] (-1,-1) circle (3pt);
        \filldraw[blue] (-1,1) circle (3pt);
        
    \end{tikzpicture}
    \caption{Subdifferentials of $f(\vec{x}) = \|\vec{x}\|_1$.}
    \label{fig:l1_dual}
\end{subfigure}

\hfill
\caption{The geometry of $f(\vec{x}) = \|\vec{x}\|_1$.}
\label{fig:examples_l1}
\end{figure}

Note that the family of these sets, i.e., either the set of closures of sets of equal subdifferential or the set of subdifferentials, have the additional property that any non-empty face of the included polyhedral sets is again in the family. Furthermore, the intersection of any two non-disjoint polyhedral sets is again in the family. These properties hold for all convex piecewise linear functions and this structure is referred to as a polyhedral complex, an object commonly encountered in geometry, topology and combinatorics, for example, see \cite{ziegler2012lectures}.

\begin{definition}[polyhedral complex]\label{def:polyhedral_complex}
We call a set $\complex{E} \subset \mathcal{P}(\reals^n)$, where $\mathcal{P}(\reals^n)$ denotes the power-set of $\reals^n$, a polyhedral complex if the following holds:
\begin{itemize}
    \item All $\set{E} \in \complex{E}$ are polyhedral.
    \item If $\set{F}$ is a non-empty face of $\set{E} \in \complex{E}$, then $\set{F} \in \complex{E}$.
    \item If $\set{E}, \set{F} \in \complex{E}$ have common intersection, then $\set{E} \cap \set{F}$ is a face of both $\set{E}$ and $\set{F}$.
\end{itemize}
Note that compared to the conventional definition of a polyhedral complex, we do not require the empty set to be in $\complex{E}$.
\end{definition}

For any polyhedral set $\set{P}$, the set of all faces of $\set{P}$ forms a polyhedral complex. Therefore, the polyhedral complex can be considered as a generalization of a polyhedral set. Additional examples of polyhedral complexes are given in Figure \ref{fig:examples_sparsity_subdifferential_complexes}.

\begin{figure}
\centering
\begin{subfigure}{0.5\textwidth}
    \centering
    \begin{tikzpicture}[scale = \scalefac]
        \draw[gray] (-2,0) -- (2, 0);
        \draw[gray] (0,-2) -- (0, 2);
        
        \fill[pattern=crosshatch dots, pattern color=red] (0,0) rectangle ++(2,2);
        
        \draw[red, line width = 0.2em] (0,0) -- (2, 0);
        \draw[red, line width = 0.2em] (0,0) -- (0, 2);
        \draw[red, line width = 0.2em] (0,0) -- (2, 2);
        
        \filldraw[red] (0,0) circle (3pt);
        
    \end{tikzpicture}
    \caption{$\complex{F}_{|x_1 - x_2| + \xi(\vec{x})}$}
    \label{fig:examples_sparsity_subdifferential_complexes:sparsity_NNTV}
\end{subfigure}%
\begin{subfigure}{0.5\textwidth}
    \centering
    \begin{tikzpicture}[scale = \scalefac]
        \draw[gray] (-2,0) -- (2, 0);
        \draw[gray] (0,-2) -- (0, 2);
        
        \fill[pattern=crosshatch dots, pattern color=red] (-1,-1) rectangle ++(2,2);
        \fill[pattern=crosshatch dots, pattern color=red] (1,1) -- (-1, 1) -- (-2, 2) -- (2, 2) -- cycle;
        \fill[pattern=crosshatch dots, pattern color=red] (-1,1) -- (-1, -1) -- (-2, -2) -- (-2, 2) -- cycle;
        \fill[pattern=crosshatch dots, pattern color=red] (-1,-1) -- (1, -1) -- (2, -2) -- (-2, -2) -- cycle;
        \fill[pattern=crosshatch dots, pattern color=red] (1,-1) -- (1, 1) -- (2, 2) -- (2, -2) -- cycle;
        
        \draw[red, line width = 0.2em] (1,1) -- (1, -1) -- (-1, -1) -- (-1, 1) -- cycle;
        \draw[red, line width = 0.2em] (1,1) -- (2,2);
        \draw[red, line width = 0.2em] (-1,1) -- (-2,2);
        \draw[red, line width = 0.2em] (1,-1) -- (2,-2);
        \draw[red, line width = 0.2em] (-1,-1) -- (-2,-2);
        
        \filldraw[red] (1,1) circle (3pt);
        \filldraw[red] (-1,1) circle (3pt);
        \filldraw[red] (1,-1) circle (3pt);
        \filldraw[red] (-1,-1) circle (3pt);
        
    \end{tikzpicture}
    \caption{$\complex{F}_{\max\{1,\|\vec{x}\|_{\infty}\}}$}
    \label{fig:examples_sparsity_subdifferential_complexes:sparsity_linf}
\end{subfigure}
\hfill
\begin{subfigure}{0.5\textwidth}
    \centering
    \begin{tikzpicture}[scale = \scalefac]
        \draw[gray] (-2,0) -- (2, 0);
        \draw[gray] (0,-2) -- (0, 2);
        
        \fill[pattern=crosshatch dots, pattern color=blue] (-2,1) -- (-1, 1) -- (1, -1) -- (1, -2) -- (-2, -2) -- cycle;
        
        \draw[blue, line width = 0.2em] (-2,1) -- (-1, 1) -- (1, -1) -- (1, -2);
        
        \filldraw[blue] (-1,1) circle (3pt);
        \filldraw[blue] (1,-1) circle (3pt);
        
    \end{tikzpicture}
    \caption{$\complex{F}_{|x_1 - x_2| + \xi(\vec{x})}^\star$}
    \label{fig:examples_sparsity_subdifferential_complexes:subdifferential_NNTV}
\end{subfigure}%
\begin{subfigure}{0.5\textwidth}
\centering
    \begin{tikzpicture}[scale = \scalefac]
        \draw[gray] (-2,0) -- (2, 0);
        \draw[gray] (0,-2) -- (0, 2);
        
        \fill[pattern=crosshatch dots, pattern color=blue] (0,0) -- (2, 0) -- (0, 2)  -- cycle;
        \fill[pattern=crosshatch dots, pattern color=blue] (0,0) -- (0, 2) -- (-2, 0)  -- cycle;
        \fill[pattern=crosshatch dots, pattern color=blue] (0,0) -- (-2, 0) -- (0, -2)  -- cycle;
        \fill[pattern=crosshatch dots, pattern color=blue] (0,0) -- (0, -2) -- (2, 0)  -- cycle;
        
        \draw[blue, line width = 0.2em] (2,0) -- (0, 2) -- (-2, 0) -- (0, -2) -- cycle;
        \draw[blue, line width = 0.2em] (2,0) -- (-2,0);
        \draw[blue, line width = 0.2em] (0,2) -- (0,-2);
        
        \filldraw[blue] (0,0) circle (3pt);
        \filldraw[blue] (2,0) circle (3pt);
        \filldraw[blue] (0,2) circle (3pt);
        \filldraw[blue] (-2,0) circle (3pt);
        \filldraw[blue] (0,-2) circle (3pt);
        
    \end{tikzpicture}
    \caption{$\complex{F}_{\max\{1,\|\vec{x}\|_{\infty}\}}^\star$}
    \label{fig:examples_sparsity_subdifferential_complexes:subdifferential_linf}
\end{subfigure}
\hfill
\caption{Examples of primal complexes and their corresponding dual complexes. For clearity, function $\xi(\vec{x})$ denotes the characterstic function of the nonnegative orthant $\chi_{\reals^n_{\geq 0}}(\vec{x}).$}
\label{fig:examples_sparsity_subdifferential_complexes}
\end{figure}
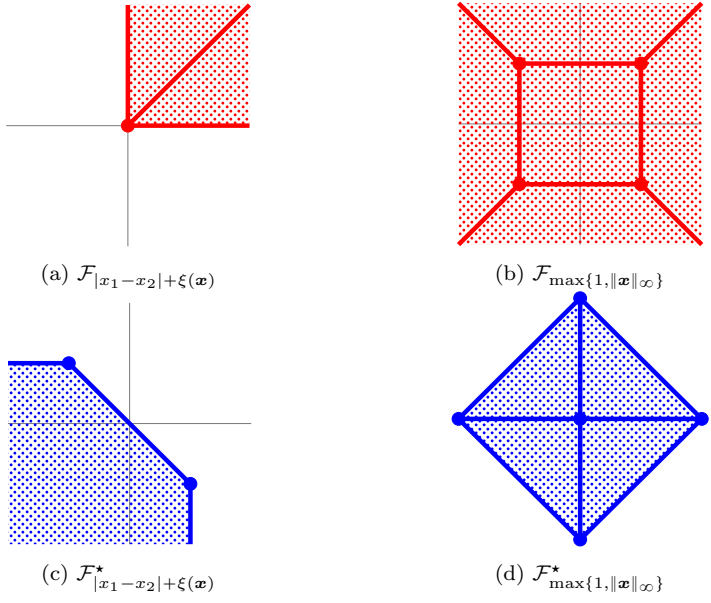

Instead of defining a polyhedral complex through the closure of sets of constant subdifferential, we can define the same polyhedral complex by projecting the polyhedral epigraph of a convex piecewise linear function onto the domain in a manner that preserves the polyhedral complex created by the faces of the epigraph.

Similarly, one of the definitions given for convex piecewise linear functions is that the domain can be covered by finitely many polyhedral sets, such that the restriction of $f$ to each of these polyhedral sets is an affine function. The smallest polyhedral complex describing such covering can be constructed as in the following definition.

\begin{definition}[primal complex]\label{def:domain_complex}
Let $f : \reals^{n} \rightarrow \reals \cup \{\infty\}$ be convex piecewise linear. Define the lower faces of the epigraph of $f$ as
\begin{equation*}
    \bar{\complex{F}}_f := \{\set{F}\ \ntext{is a face of}\ \epi(f)\,|\, \set{F} \neq \emptyset\ \ntext{and}\ \dim(\set{F}) \neq n + 1\}.
\end{equation*}

Let $P : \reals^{n+1} \rightarrow \reals^n$ be the linear projection map $P(\vec{x}, \vec{t}) = \vec{x}$ and define the primal complex of $f$ as
\begin{equation*}
    \complex{F}_f := P\bar{\complex{F}}_f = \{P\bar{\set{F}}\,|\, \bar{\set{F}} \in \bar{\complex{F}}_f\}.
\end{equation*}
\end{definition}

Intuitively, one can consider $\complex{F}_f$ as projecting $\epi(f)$ or $\graph(f)$ back onto the domain, but keeping the structure of the faces of $\epi(f)$, as illustrated in Figure \ref{fig:construction_regular_subdivision}. Note that Figure \ref{fig:construction_regular_subdivision:2D} corresponds to the complex shown in \ref{fig:examples_sparsity_subdifferential_complexes:sparsity_linf}.

\input{tikz_figures/fig_regular_subdivision}

Though often referred to as a regular subdivision of $\dom(f)$, see for example \cite{ziegler2012lectures}, we use the term primal complex due to duality properties described in Subsection \ref{subsec:duality}. 

Constructing the primal complex as a projection of a polyhedral set will allow us to derive various properties through the properties of the polyhedral epigraph. The following proposition states that the primal complex corresponds to the closure of sets of equal subdifferential.

\begin{proposition}[constant subdifferential]\label{prop:constant_subdifferential}
Let $f : \reals^{n} \rightarrow \reals \cup \{\infty\}$ be convex piecewise linear. Each polyhedral set $\set{F} \in \complex{F}_f$ can be characterized as follows. Let $\vec{x} \in \relint(\set{F})$ be arbitrary, then the relative interior $\relint(\set{F})$ consists precisely of all $\vec{y}\in \reals^n$ for which $\partial f(\vec{x}) = \partial f(\vec{y})$.

The constant subdifferential corresponding to $\relint(\set{F})$ will be denoted by $\partial f_{\set{F}}$.
\end{proposition}
\begin{proof}
Consider the polyhedral set $\epi(f)$. The interior of each face of $\epi(f)$ is characterized by its normal cone. Furthermore, the normal cone of a lower face of $\epi(f)$ characterizes the subdifferential of $f$ at the interior of the face projected back onto coordinate space.
\hfill$\square$\end{proof}

Because these sets will describe the low-dimensional sets of interest, an alternative name used in the literature is pattern, see for example \cite{hejny2024unveiling}.

The polyhedral complex corresponding to the subdifferentials is easier to define.
\begin{definition}[dual complex]\label{def:subdifferential_complex}
Let $f : \reals^{n} \rightarrow \reals \cup \{\infty\}$ be convex piecewise linear. Define the corresponding dual complex as 
\begin{equation*}
    \complex{F}_f^\star := \{\partial f_\set{F}\,|\, \set{F} \in \complex{F}_f\} = \{\partial f(\vec{x})\,|\, \vec{x} \in \reals^n\}.
\end{equation*}
\end{definition}

Examples of primal and dual complexes can be found in Figures \ref{fig:examples_l1} and \ref{fig:examples_sparsity_subdifferential_complexes}.

Thus we can associate a pair of polyhedral complexes $(\complex{F}_f, \complex{F}_f^\star)$ and a bijection $\partial f : \complex{F}_f \rightarrow \complex{F}_f^\star$ between the primal complex and the dual complex with any convex piecewise linear function $f$. The next subsection is dedicated to the relation between the primal complex $\complex{F}_f$ and the dual complex $\complex{F}_f^\star$.

\subsection{Duality}\label{subsec:duality}

If $\set{P}$ is a polyhedral set with face $\set{F}$ then the normal cone $\set{N}_\set{P}(\vec{x})$ at any point $\vec{x}$ in the relative interior of $\set{F}$ is the same, which we can write simply as $\set{N}_\set{P}(\relint(\set{F}))$. For two faces $\set{F}$ and $\set{H}$ such that $\set{F}$ is a face of $\set{H}$, their corresponding normal cones $\set{N}_\set{P}(\relint(\set{F}))$ and $\set{N}_\set{P}(\relint(\set{H}))$ have the reverse relation, that $\set{N}_\set{P}(\relint(\set{H}))$ is a face of $\set{N}_\set{P}(\relint(\set{F}))$. This property directly extends to the pair $(\complex{F}_f, \complex{F}_f^\star)$ by slicing the normal cones of the polyhedral epigraph.

\begin{lemma}\label{lemma:dual_inclusion}
Let $f : \reals^{n} \rightarrow \reals \cup \{\infty\}$ be convex piecewise linear. Let $\set{F}, \set{H} \in \complex{F}_f$ with corresponding $\partial f_{\set{F}}, \partial f_{\set{H}} \in \complex{F}_f^\star$, then $\set{F}$ is a face of $\set{H}$ if and only if $\partial f_\set{H}$ is a face of $\partial f_{\set{F}}$. 
\end{lemma}

The primal complex $\complex{F}_f$ and dual complex $\complex{F}_f^\star$ are actually dual through conjugacy of the convex piecewise linear function $f$. Recall that the convex conjugate of a convex function $f$ is another convex function defined by $f^\star(\vec{y}) := \sup_{\vec{x}\in\reals^n}\{\vec{y}^T\vec{x} - f(\vec{x})\}$. Furthermore, the convex conjugate of a convex piecewise linear function is again convex piecewise linear. The following proposition is both useful for deriving properties and for computing either complexes.

\begin{proposition}[Duality of complexes]\label{prop:duality_of_complexes}
Let $f : \reals^{n} \rightarrow \reals \cup \{\infty\}$ be convex piecewise linear. Then,
\begin{equation*}
    \left(\complex{F}_{f^\star}, \complex{F}_{f^\star}^\star\right) = \left(\complex{F}_f^\star, \complex{F}_f\right).
\end{equation*}
\end{proposition}
\begin{proof}
Let $g = f^\star$, then $g^\star = f$ by $f$ being closed and the Fenchel-Moreau theorem \cite[Corollary 13.38]{bauschke2011convex}. Therefore, $\complex{F}_f^\star = \complex{F}_{f^\star}$ is equivalent to $\complex{F}_{g^\star}^\star = \complex{F}_g$. Thus, we only need to prove $\complex{F}_{f^\star}^\star = \complex{F}_{f}$ for any convex piecewise linear function $f$.

A fundamental property of proper, closed, convex functions is the equivalence of the Fenchel-Young inequality \cite[Proposition 16.10]{bauschke2011convex}, i.e., if $\vec{x} \in \dom(f)$ and $\vec{y} \in \dom(f^\star)$, then
\begin{equation*}
    \vec{y} \in \partial f(\vec{x}) \ntext{ if and only if }\vec{x} \in \partial f^\star(\vec{y}).
\end{equation*}

Let $\set{F} \in \complex{F}_f$ and $\partial f_{\set{F}} \in \complex{F}_f^\star$. We need to show that for all $\vec{y} \in \partial f_{\set{F}}$, we have $\partial f^\star(\vec{y}) = \set{F}$.

Let $\vec{y} \in \partial f_{\set{F}}$ and $\vec{x} \in \relint(\set{F})$, then $\vec{y} \in \partial f_{\set{F}} = \partial f(\vec{x})$, thus $\vec{x} \in \partial f^\star(\vec{y})$, hence $\relint(\set{F}) \subseteq \partial f^\star(\vec{y})$. Because the subdifferential is a closed set we get $\set{F} \subseteq \partial f^\star(\vec{y})$.

Let $\vec{x} \in \set{F}$ but $\vec{x} \not\in \partial f^\star(\vec{y})$. Then $\vec{y} \not\in \partial f(\vec{x})$. But by assumption $\vec{y} \in \partial f_{\set{F}} \subseteq \partial f(\vec{x})$, which is a contradiction. 
\hfill$\square$\end{proof}

An example of this conjugacy is as follows. Let $\set{V} \subset \reals^n$ be a finite set of points and define the support function $\sigma_{\set{V}}(\vec{x}) := \sup_{\vec{v}\in \set{V}}\{\vec{v}^T\vec{x}\}$. Then $\chi_{\conv(\set{V})}^\star = \sigma_{\set{V}}$. Therefore, $\complex{F}_{\chi_{\conv(\set{V})}}^\star$ corresponds to the polyhedral set $\conv(\set{V})$ with all of its faces. The case $\|\vec{x}\|_1$ corresponds to $\set{V} = \{(1,1), (-1,1), (1,-1), (-1,-1)\}$, as shown in Figure \ref{fig:examples_l1}.

Some useful duality properties are given in the following lemma.

\begin{lemma}\label{lemma:properties_duality_of_complexes_2}
Let $f : \reals^{n} \rightarrow \reals \cup \{\infty\}$ be convex piecewise linear. Let $\set{F} \in \complex{F}_f$ with corresponding $\partial f_{\set{F}} \in \complex{F}_f^\star$, then the following holds:
\begin{enumerate}[label=\roman*]
    \item $\dim(\set{F}) + \dim(\partial f_{\set{F}}) = n$ .\label{lemma:duality_dimensions}
    \item $\partial f_{\set{F}} \bot \set{F}$, i.e., $(\partial f_{\set{F}} - \partial f_{\set{F}})^T(\set{F} - \set{F}) = \{\vec{0}\}$. \label{lemma:duality_orthogonal}
    \item $\set{F} + \partial f_{\set{F}}^\bot \in \aff(\set{F})$. \label{lemma:duality_orthogonal_2}
\end{enumerate}
\end{lemma}

\begin{proof}
(\ref{lemma:duality_dimensions}) Follows from the normal cone construction.

(\ref{lemma:duality_orthogonal}) Let $\vec{x}, \vec{x}' \in \set{F}$ and $\vec{y}, \vec{y}' \in \partial f_{\set{F}}$, then
\begin{equation*}
    (\vec{y} - \vec{y}')^T(\vec{x} - \vec{x}') = \vec{y}^T\vec{x} + \vec{y}'^T\vec{x}' - \vec{y}'^T\vec{x} - \vec{y}^T\vec{x}'. 
\end{equation*}
By conjugacy, $\vec{y}^T\vec{x} = f^\star(\vec{y}) + f(\vec{x})$ and similarly for the other terms. All thee terms will cancel out resulting in
\begin{equation*}
    (\vec{y} - \vec{y}')^T(\vec{x} - \vec{x}') = \vec{0}.
\end{equation*}
    
(\ref{lemma:duality_orthogonal_2}) Follows from properties \ref{lemma:duality_dimensions} and \ref{lemma:duality_orthogonal}.

\hfill$\square$\end{proof}

The various properties from Lemma \ref{lemma:properties_duality_of_complexes_2} are visualized in Figure \ref{fig:examples_duality_complexes}.

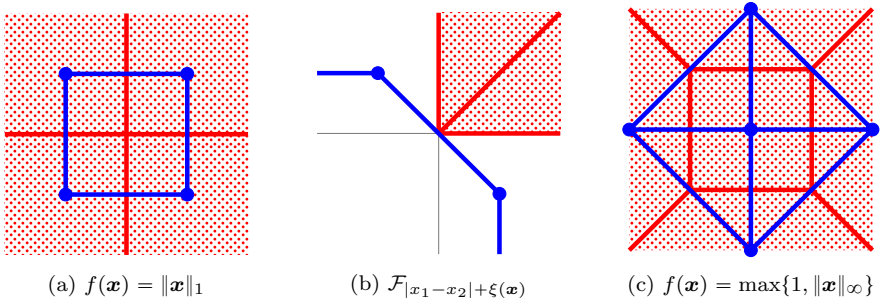
\begin{figure}
\centering
\begin{subfigure}{0.3\textwidth}
    \centering
    \begin{tikzpicture}[scale = \scalefac]
        \draw[gray] (-2,0) -- (2, 0);
        \draw[gray] (0,-2) -- (0, 2);
        
        \fill[pattern=crosshatch dots, pattern color=red] (0,0) rectangle ++(-2,-2);
        \fill[pattern=crosshatch dots, pattern color=red] (0,0) rectangle ++(2,2);
        
        \fill[pattern=crosshatch dots, pattern color=red] (0,0) rectangle ++(-2,2);
        \fill[pattern=crosshatch dots, pattern color=red] (0,0) rectangle ++(2,-2);
        
        \draw[red, line width = 0.2em] (0,0) -- (2, 0);
        \draw[red, line width = 0.2em] (0,0) -- (0, 2);
        \draw[red, line width = 0.2em] (0,0) -- (-2, 0);
        \draw[red, line width = 0.2em] (0,0) -- (0, -2);

        \draw[blue, line width = 0.2em] (1,1) -- (1, -1) -- (-1, -1) -- (-1, 1) -- cycle;
        
        \filldraw[blue] (1,1) circle (3pt);
        \filldraw[blue] (1,-1) circle (3pt);
        \filldraw[blue] (-1,-1) circle (3pt);
        \filldraw[blue] (-1,1) circle (3pt);
        
    \end{tikzpicture}
    \caption{$f(\vec{x}) = \|\vec{x}\|_1$}
    \label{fig:examples_duality:sparsity_l1}
\end{subfigure}
\hfill
\begin{subfigure}{0.3\textwidth}
    \centering
    \begin{tikzpicture}[scale = \scalefac]
        \draw[gray] (-2,0) -- (2, 0);
        \draw[gray] (0,-2) -- (0, 2);
        
        \fill[pattern=crosshatch dots, pattern color=red] (0,0) rectangle ++(2,2);
        
        \draw[red, line width = 0.2em] (0,0) -- (2, 0);
        \draw[red, line width = 0.2em] (0,0) -- (0, 2);
        \draw[red, line width = 0.2em] (0,0) -- (2, 2);

        \draw[blue, line width = 0.2em] (-2,1) -- (-1, 1) -- (1, -1) -- (1, -2);
        
        \filldraw[blue] (-1,1) circle (3pt);
        \filldraw[blue] (1,-1) circle (3pt);

    \end{tikzpicture}
    \caption{$\complex{F}_{|x_1 - x_2| + \xi(\vec{x})}$}
    \label{fig:examples_duality:sparsity_NNTV}
\end{subfigure}
\hfill
\begin{subfigure}{0.3\textwidth}
    \centering
    \begin{tikzpicture}[scale = \scalefac]
        \draw[gray] (-2,0) -- (2, 0);
        \draw[gray] (0,-2) -- (0, 2);
        
        \fill[pattern=crosshatch dots, pattern color=red] (-1,-1) rectangle ++(2,2);
        \fill[pattern=crosshatch dots, pattern color=red] (1,1) -- (-1, 1) -- (-2, 2) -- (2, 2) -- cycle;
        \fill[pattern=crosshatch dots, pattern color=red] (-1,1) -- (-1, -1) -- (-2, -2) -- (-2, 2) -- cycle;
        \fill[pattern=crosshatch dots, pattern color=red] (-1,-1) -- (1, -1) -- (2, -2) -- (-2, -2) -- cycle;
        \fill[pattern=crosshatch dots, pattern color=red] (1,-1) -- (1, 1) -- (2, 2) -- (2, -2) -- cycle;
        
        \draw[red, line width = 0.2em] (1,1) -- (1, -1) -- (-1, -1) -- (-1, 1) -- cycle;
        \draw[red, line width = 0.2em] (1,1) -- (2,2);
        \draw[red, line width = 0.2em] (-1,1) -- (-2,2);
        \draw[red, line width = 0.2em] (1,-1) -- (2,-2);
        \draw[red, line width = 0.2em] (-1,-1) -- (-2,-2);

        \draw[blue, line width = 0.2em] (2,0) -- (0, 2) -- (-2, 0) -- (0, -2) -- cycle;
        \draw[blue, line width = 0.2em] (2,0) -- (-2,0);
        \draw[blue, line width = 0.2em] (0,2) -- (0,-2);
        
        \filldraw[blue] (0,0) circle (3pt);
        \filldraw[blue] (2,0) circle (3pt);
        \filldraw[blue] (0,2) circle (3pt);
        \filldraw[blue] (-2,0) circle (3pt);
        \filldraw[blue] (0,-2) circle (3pt);
        
    \end{tikzpicture}
    \caption{$f(\vec{x}) = \max\{1,\|\vec{x}\|_{\infty}\}$}
    \label{fig:examples_duality:sparsity_linf}
\end{subfigure}
\hfill
\caption{Examples of the duality between parts of the primal complexes (red) and their corresponding dual complexes (blue).}
\label{fig:examples_duality_complexes}
\end{figure}

\subsection{Total Variation regularization on graphs}\label{subsec:TV}

A popular choice for regularization is total variation, which is used to penalize jumps and thereby promote piecewise constant signals, see for example \cite{condat2017discrete,kolmogorov2016total}. We will study some properties of vertices of the dual complex of total variation regularization on graphs, which we will need in Section \ref{sec:implications}.

Let $G = (\set{V}, \set{E})$ be an undirected graph with nodes $\set{V} = \{v_i\}_{i = 1}^{\#\set{V}}$ and links $\set{E} \subset \{\{v_i, v_j\}\,|\, v_i, v_j \in \set{V}, v_i \neq v_j\}$. We consider a parameter vector $\vec{x} \in \mathbb{R}^{\#\set{V}}$ which assigns a value to each node in the graph $G$. For each link $\{v_i, v_j\} \in \set{E}$, we will penalize the difference of the values at the endpoints with the penalty $|x_i - x_j|$. Summing over all of these penalties results in the total variation (TV) regularization, written as,
\begin{equation*} \label{eq:TV_on_graph}
    \text{TV}_{G}(\vec{x}) := \sum_{\{v_i, v_j\} \in \set{E}}|x_i - x_j|.
\end{equation*}
Alternatively, denote by $\vec{d}_{i,j} \in \mathbb{R}^{\#\set{V}}$ the vector that is $1$ at the $i$-th index, $-1$ at the $j$-th index, and $0$ otherwise. Then we can equivalently write the TV regularization as,
\begin{equation} \label{eq:TV_on_graph_l1}
    \text{TV}_{G}(\vec{x}) := \sum_{\{v_i, v_j\} \in \set{E}}|\vec{d}_{i,j}^T\vec{x}| = \|D\vec{x}\|_1,
\end{equation}
where the rows of $D \in \mathbb{R}^{\#\set{E}\times\#\set{V}}$ are the vectors $\vec{d}_{i,j}^T$ for all $\{v_i, v_j\} \in \set{E}$.

A single term $|\vec{d}_{i,j}^T\vec{x}|$ is non-differentiable if and only if $x_i = x_j$, thus non-differentiabillity corresponds to vectors $\vec{x}$ such that some neighbouring nodes are assigned the same value. If $x_i \neq x_j$, then $|\vec{d}_{i,j}^T\vec{x}|$ is differentiable with gradient $\vec{d}_{i, j}$ if the value is increasing from node $i$ to $j$, i.e., $x_i < x_j$, and $-\vec{d}_{i,j}$ if the value is decreasing from node $i$ to $j$, i.e., $x_i > x_j$. Thus, for a single term $|\vec{d}_{i,j}^T\vec{x}|$, the dual complex is the polytope with vertices $\vec{d}_{i,j}$ and $-\vec{d}_{i,j}$. Because the total variation regularization \eqref{eq:TV_on_graph_l1} is the sum over these terms, the dual complex is the polytope obtained by taking the Minkowski sum over all these smaller polytopes. Thus, the dual complex of $\text{TV}_G$ is the polytope
\begin{equation} \label{eq:TV_polytope}
    \set{P}_{\text{TV}_G} = \conv\left(\left\{\sum_{\{v_i, v_j\} \in \set{E}} u_{i,j}\vec{d}_{i,j}\,|\, \vec{u} \in \{-1,1\}^{\#\set{E}}\right\}\right),
\end{equation}
with all of its faces. The vector $\vec{u}$ in \eqref{eq:TV_polytope} can be interpreted as assigning directions to the links in the graph $G$, with the product $u_{i,j}\vec{d}_{i,j}$ being $1$ if the link is directed from $j$ to $i$ and $-1$ if the link is directed in the other direction. The total sum $\sum_{\{v_i, v_j\} \in \set{E}}u_{i,j}\vec{d}_{i,j}$ can then be interpreted as the differences between the in-degree and out-degree at each node. However, not all possible direction assignments $\vec{u}$ result in vertices of the polytope $\set{P}_{\text{TV}_G}$. The following proposition shows that only those direction assignments $\vec{u}$ such that the associated directed graph is acyclic describe vertices.

\begin{proposition}\label{prop:acyclic_vertex}
    The vertices of the polytope $\set{P}_{\text{TV}_G}$ are the vectors $\vec{z}_{\vec{u}} = \sum_{\{v_i, v_j\} \in \set{E}} u_{i,j}\vec{d}_{i,j}$ with $\vec{u} \in \{-1,1\}^{\#\set{E}}$ such that the associated directed graph is acyclic.
\end{proposition}
\begin{proof}

    Let $\vec{z}_{\vec{u}}$ such that the associated directed graph is cyclic. Because $\vec{z}_{\vec{u}} \in \set{P}_{\text{TV}_G}$, there must exists $\vec{x}$ such that $\vec{z}_{\vec{u}} \in \partial \text{TV}_G (\vec{x})$, more specifically, the values in $\vec{x}$ must strictly increase over any directed link in the directed graph. Because the directed graph is cyclic, this results in a contradiction, hence the directed graph associated to $\vec{z}_{\vec{u}}$ is acyclic.

    Let $\vec{z}_{\vec{u}}$ such that the associated directed graph is acyclic. Construct $\vec{x} \in \mathbb{R}^{\#\set{V}}$ inductively as follows. For each node $i$ with no in-going links, let $x_i = 0$. For all other nodes $j$, let $x_j$ be the one larger than the maximum over the $x_i$ from incoming nodes. By construction, $\text{TV}_G$ is differentiable at $\vec{x}$ with gradient $\vec{z}_{\vec{u}}$, hence $\vec{z}_{\vec{u}}$ is a vertex of the polytope $\set{P}_{\text{TV}_G}$.
\hfill$\square$\end{proof}

\begin{corollary}
    An undirected graph $G$ is a tree, i.e., the graph does not have cycles, if and only if, the vertices of $\set{P}_{\text{TV}_G}$ are all the vectors of the form $\sum_{\{v_i, v_j\} \in \set{E}}u_{i,j}\vec{d}_{i,j}$ with $\vec{u} \in \{-1,1\}^{\#\set{E}}$.
\end{corollary}

In many applications, negative components in the parameters $\vec{x}$ make no physical sense, and nonnegativity constraints are added. When nonnegativity constraints are added to total variation regularization, i.e.,
\begin{equation*}
    \text{NN-TV}_{G}(\vec{x}) := \text{TV}_G(\vec{x}) + \chi_{\mathbb{R}_{\geq 0}^{\#\set{V}}}(x),
\end{equation*}
then its dual complex is described by the Minkowski sum $\set{P}_{\text{TV}_G} - \mathbb{R}_{\geq 0}^{\#\set{V}}$, the vertices of which are the same as $\set{P}_{\text{TV}_G}$, as shown in the following proposition.

\begin{proposition}\label{prop:acyclic_vertex_NN}
    The vertices of $\set{P}_{\text{TV}_G} - \mathbb{R}_{\geq 0}^{\#\set{V}}$ and $\set{P}_{\text{TV}_G}$ are the same.
\end{proposition}
\begin{proof}
    If $\vec{v}$ be a vertex of $\set{P}_{\text{TV}_G}$ that is not a vertex of $\set{P}_{\text{TV}_G} - \mathbb{R}_{\geq 0}^{\#\set{V}}$, then $\vec{v}$ can be written as a convex combination of the other vertices of $\set{P}_{\text{TV}_G}$ plus a non-zero vector $\vec{p}$ from the cone $-\mathbb{R}_{\geq 0}^{\#\set{V}}$. However, every vector $\vec{x} \in \set{P}_{\text{TV}_G}$ satisfies $\vec{1}^T\vec{x} = 0$, but $\vec{1}^T\vec{v} = \vec{1}^T\vec{p} < 0$, a contradiction. Thus every vertex of $\set{P}_{\text{TV}_G}$ is a vertex of $\set{P}_{\text{TV}_G} - \mathbb{R}_{\geq 0}^{\#\set{V}}$. 
    
    For the converse, note that  vertices of the Minkowski sum of a polytope and a polyhedral cone are all vertices of the polytope.
\hfill$\square$\end{proof}

\subsection{Level sets and linear subspaces} \label{subsec:level_sets_linear_subspaces}

The orthogonality properties from Lemma \ref{lemma:properties_duality_of_complexes_2} are fundamental in the discussion of the uniqueness of the solution of the sparse regularized linear regression problem \eqref{eq:variational_introduction}. Note particularly that if $\set{F} \in \complex{F}_f$ does not consist of a single point and $\nullspace(A) \cap \partial f_{\set{F}}^\bot \neq \emptyset$, then $\relint(\set{F})$ contains an infinite subset $\set{L}$ such that $A\vec{x} = A\vec{y}$ for all $\vec{x}, \vec{y} \in \set{L}$. Thus, either $\set{L}$ does not contain the solution of the optimization problem, or all the points in $\set{L}$ are solutions. Note that, if they are all solutions, then $\set{L}$ must lie in a level set of $f$. Each level set of a convex piecewise linear function can be given the structure of a polyhedral complex.

\begin{definition}[level set complex]\label{def:level_set_complex}
Let $f: \reals^{n} \rightarrow \reals \cup \{\infty\}$ be convex piecewise linear. For every $c \in \reals$, define the level set $\complex{L}_f(c)$ as the polyhedral complex corresponding to the level set $\{\vec{x}\in \reals^n\,|\, f(\vec{x}) = c\}$.
\end{definition}

\begin{figure}
\centering
\begin{subfigure}{0.3\textwidth}
    \centering
    \begin{tikzpicture}[scale = \scalefac]
        \draw[gray] (-2,0) -- (2, 0);
        \draw[gray] (0,-2) -- (0, 2);
        
        \fill[pattern=crosshatch dots, pattern color=red] (0,0) rectangle ++(-2,-2);
        \fill[pattern=crosshatch dots, pattern color=red] (0,0) rectangle ++(2,2);
        
        \fill[pattern=crosshatch dots, pattern color=red] (0,0) rectangle ++(-2,2);
        \fill[pattern=crosshatch dots, pattern color=red] (0,0) rectangle ++(2,-2);
        
        \draw[red, line width = 0.2em] (0,0) -- (2, 0);
        \draw[red, line width = 0.2em] (0,0) -- (0, 2);
        \draw[red, line width = 0.2em] (0,0) -- (-2, 0);
        \draw[red, line width = 0.2em] (0,0) -- (0, -2);
        
        \filldraw[red] (0,0) circle (3pt);
        
        \draw[green, line width = 0.2em] (1,0) -- (0, -1) -- (-1, 0) -- (0, 1) -- cycle;
        
        \filldraw[green] (1,0) circle (3pt);
        \filldraw[green] (-1,0) circle (3pt);
        \filldraw[green] (0,1) circle (3pt);
        \filldraw[green] (0,-1) circle (3pt);
        
    \end{tikzpicture}
    \caption{$\complex{F}_{\|\vec{x}\|_1}$}
    \label{fig:examples_level_sets:sparsity_l1}
\end{subfigure}
\hfill
\begin{subfigure}{0.3\textwidth}
    \centering
    \begin{tikzpicture}[scale = \scalefac]
        \draw[gray] (-2,0) -- (2, 0);
        \draw[gray] (0,-2) -- (0, 2);
        
        \fill[pattern=crosshatch dots, pattern color=red] (0,0) rectangle ++(2,2);
        
        \draw[red, line width = 0.2em] (0,0) -- (2, 0);
        \draw[red, line width = 0.2em] (0,0) -- (0, 2);
        \draw[red, line width = 0.2em] (0,0) -- (2, 2);
        
        \filldraw[red] (0,0) circle (3pt);
        
        \draw[green, line width = 0.2em] (0,0.5) -- (1.5, 2);
        \draw[green, line width = 0.2em] (0.5,0) -- (2, 1.5);
        
        \filldraw[green] (0,0.5) circle (3pt);
        \filldraw[green] (0.5,0) circle (3pt);

    \end{tikzpicture}
    \caption{$\complex{F}_{|x_1 - x_2| + \xi(\vec{x})}$}
    \label{fig:examples_level_sets:sparsity_NNTV}
\end{subfigure}
\hfill
\begin{subfigure}{0.3\textwidth}
    \centering
    \begin{tikzpicture}[scale = \scalefac]
        \draw[gray] (-2,0) -- (2, 0);
        \draw[gray] (0,-2) -- (0, 2);
        
        \fill[pattern=crosshatch dots, pattern color=red] (-1,-1) rectangle ++(2,2);
        \fill[pattern=crosshatch dots, pattern color=red] (1,1) -- (-1, 1) -- (-2, 2) -- (2, 2) -- cycle;
        \fill[pattern=crosshatch dots, pattern color=red] (-1,1) -- (-1, -1) -- (-2, -2) -- (-2, 2) -- cycle;
        \fill[pattern=crosshatch dots, pattern color=red] (-1,-1) -- (1, -1) -- (2, -2) -- (-2, -2) -- cycle;
        \fill[pattern=crosshatch dots, pattern color=red] (1,-1) -- (1, 1) -- (2, 2) -- (2, -2) -- cycle;
        
        \draw[red, line width = 0.2em] (1,1) -- (1, -1) -- (-1, -1) -- (-1, 1) -- cycle;
        \draw[red, line width = 0.2em] (1,1) -- (2,2);
        \draw[red, line width = 0.2em] (-1,1) -- (-2,2);
        \draw[red, line width = 0.2em] (1,-1) -- (2,-2);
        \draw[red, line width = 0.2em] (-1,-1) -- (-2,-2);
        
        \filldraw[red] (1,1) circle (3pt);
        \filldraw[red] (-1,1) circle (3pt);
        \filldraw[red] (1,-1) circle (3pt);
        \filldraw[red] (-1,-1) circle (3pt);
        
        \draw[green, line width = 0.2em] (1.5,1.5) -- (-1.5,1.5) -- (-1.5,-1.5) -- (1.5,-1.5) -- cycle;
        
        \filldraw[green] (1.5,1.5) circle (3pt);
        \filldraw[green] (-1.5,1.5) circle (3pt);
        \filldraw[green] (-1.5,-1.5) circle (3pt);
        \filldraw[green] (1.5,-1.5) circle (3pt);
        
    \end{tikzpicture}
    \caption{$\complex{F}_{\max\{1,\|\vec{x}\|_{\infty}\}}$}
    \label{fig:examples_level_sets:sparsity_linf}
\end{subfigure}
\caption{Examples of primal complexes (red) and a level set complex (green).}
\label{fig:examples_level_sets}
\end{figure}
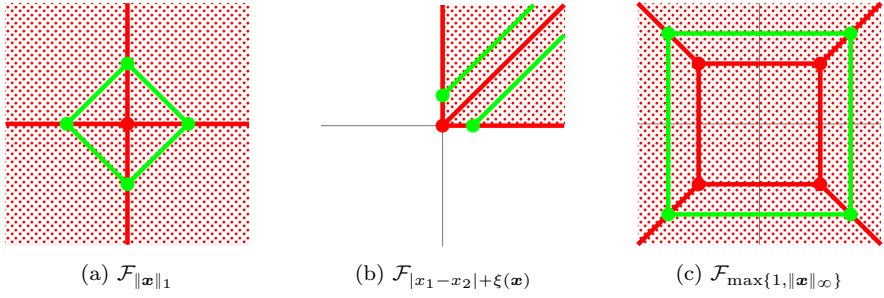

Some level set complexes are illustrated in Figure \ref{fig:examples_level_sets}. Note that each polyhedral set in a level set complex is itself a subset of a polyhedral set in the primal complex. Furthermore, the orthogonality to the dual complex extends to the level set complexes. These properties are summarized in the following two Lemmas.

\begin{lemma}\label{lem:levelsets_subsets}
Let $f: \reals^{n} \rightarrow \reals \cup \{\infty\}$ be convex piecewise linear. If $\set{L} \in \complex{L}_f(c)$ for some $c\in\reals$ and $\vec{x} \in \relint(\set{L})$ and $\vec{x} \in \set{F} \in \complex{F}$, then $\set{L} \subseteq \set{F}$.
\end{lemma}
\begin{proof}
Follows from the fact that the level set complex is part of the boundary of a polyhedral sublevel set of the epigraph.
\hfill$\square$\end{proof}

\begin{lemma}[Subdifferential orthogonal to level set]\label{lem:subdifferential_orthogonal_level_set}
Let $f: \reals^{n} \rightarrow \reals \cup \{\infty\}$ be convex piecewise linear. For every $\set{L} \in \complex{L}_f(c)$ with $c\in\reals$, let $\vec{x}\in\relint(\set{L})$, then $\partial f(\vec{x})\bot \set{L}$, i.e., $\partial f(\vec{x})^T(\set{L}-\set{L}) = \{\vec{0}\}$.

Particularly, $\set{L} + \partial f(\vec{x})^\bot \in \aff(\set{L})$.

\end{lemma}
\begin{proof}
Let $\vec{y}\in \partial f(\vec{x})$ and $\vec{v}, \vec{w} \in \set{L}$, then by conjugacy, and $\set{L}$ being a level set,
\begin{equation*}
    \vec{y}^T(\vec{v} - \vec{w}) = f^\star(\vec{y}) - f^\star(\vec{y}) + f(\vec{v}) - f(\vec{w}) = \vec{0}.
\end{equation*}
\hfill$\square$\end{proof}

Now consider the problem of minimizing $\vec{x} \mapsto D(A\vec{x},\vec{b}) + f(\vec{x})$. For any minimizer, if we can find a vector in the null-space of $A$ such that translating the minimizer by this vector stays in the same level-set, then the solution is not unique. The following lemma states conditions to obtain such vectors.

\begin{lemma}\label{lem:codimension}
Let $f: \reals^{n} \rightarrow \reals \cup \{\infty\}$ be convex piecewise linear. Let $\partial f_{\set{F}} \in \complex{F}_f^\star$. If $\codim(\partial f_{\set{F}}) - \Rank(A) = k > 0$ then $\dim(\nullspace(A) \cap \partial f_\set{F}^\bot) = k$.
\end{lemma}
\begin{proof}
We have $\codim(\partial f_{\set{F}}) - \Rank(A) = k$ if and only if $\dim(\partial f_{\set{F}}^\bot) + \dim(\nullspace(A)) = n + k$.
Thus, because $n + k > n$, we get $\partial f_{\set{F}}^\bot$ and $\nullspace(A)$ have a space of dimension $k$ in common.
\hfill$\square$\end{proof}

Lemma \ref{lem:codimension} allows us to construct whole sets of vectors $\vec{x}$ with constant $A\vec{x}$ and $f(\vec{x})$. This uniqueness breaking property is formulated in the following proposition.

\begin{proposition}\label{prop:finding_two_solutions}
Let $f: \reals^{n} \rightarrow \reals \cup \{\infty\}$ be convex piecewise linear, $A \in \reals^{m\times n}$ and $\vec{x} \in \set{F} \in \complex{F}_f$ with $\codim(\partial f_{\set{F}}) - \Rank(A) = k > 0 $. Then there exists $\set{L} \in \complex{L}_f(f(\vec{x}))$ and a polyhedral set $\set{S} \subseteq \set{L} \subseteq \set{F}$ for which $\vec{x} \in \set{S}$, $\dim(\set{S}) = k$ and $\set{S} - \set{S} \subseteq \nullspace(A)$.

\end{proposition}
\begin{proof}
Let $\set{L} \in \complex{L}_f(f(\vec{x}))$ such that $\vec{x} \in \relint(\set{L})$. If $\codim(\partial f_{\set{F}}) - \Rank(A) = k > 0 $, then by Lemma \ref{lem:codimension} and \ref{lem:subdifferential_orthogonal_level_set}, the set $\set{S} = \left[\vec{x} + \nullspace(A) \cap \partial f_\set{F}^\bot \right] \cap \set{L}$ suffices.
\hfill$\square$\end{proof}

For the converse, we want to show that if there exists two solutions, then $\codim(\partial f_{\set{F}}) - \Rank(A) > 0$. The following proposition, based on arguments from \cite{ewald2020distribution,schneider2022geometry,tardivel2021geometry}, provides this converse, requiring a condition on the effective domain of $f$ and a condition on the two solutions, which is shown to hold in Subsection \ref{subsec:uniqueness}.

\begin{proposition}\label{prop:two_solutions_large_codimension}
Let $f: \reals^{n} \rightarrow \reals \cup \{\infty\}$ be convex piecewise linear such that $\dim(\dom(f)) = n$ and  $A \in \reals^{m\times n}$. 
If there exist distinct $\vec{x}, \vec{y} \in \reals^n$ such that $\vec{x} - \vec{y} \in \nullspace(A)$ and $\row(A) \cap \partial f(\vec{x}) \cap \partial f(\vec{y}) \neq \emptyset$, then there exist $\set{F} \in \complex{F}_f$ such that $\vec{x},\vec{y}\in \set{F}$ and $\codim(\partial f_{\set{F}}) > \Rank(A)$.
\end{proposition}
\begin{proof}
Let $\partial f_{\set{F}} \in \complex{F}_f^\star$ be the smallest face of $\partial f(\vec{x}) \cap \partial f(\vec{y})$ such that $\row(A) \cap \partial f_{\set{F}} \neq \emptyset$. Then by Lemma \ref{lemma:dual_inclusion}, $\vec{x},\vec{y}\in \set{F}$. If $\codim(\partial f_{\set{F}}) > \Rank(A)$, we are done. So assume that $\codim(\partial f_{\set{F}}) \leq \Rank(A)$, or equivalently $\dim(\set{F}) \leq \Rank(A)$, then,
\begin{equation*}
    \dim(A(\set{F}-\set{F})) < \dim(\set{F}-\set{F}) = \dim(\set{F}) \leq \Rank(A).
\end{equation*}
Thus, the strict inclusion $A(\set{F}-\set{F}) \subset \col(A)$, equivalently $\nullspace(A^T) \subset (A(\set{F}-\set{F}))^\bot$ holds. Therefore, there exists $\vec{d} \in \reals^m$ such that $A^T\vec{d} \neq \vec{0}$ and $(\set{F}-\set{F})^TA^T\vec{d} = \{\vec{0}\}$, hence $\vec{0} \neq A^T\vec{d} \in \partial f_{\set{F}}^\bot$.

By assumption, there also exists $\vec{z} \in \reals^m$ such that $A^T\vec{z} \in \partial f_{\set{F}}$.
Now consider the line $A^T\vec{z} + tA^T\vec{d}$ for $t\in \reals^n$. By construction, this line is contained in $\row(A) \cap \partial f_{\set{F}}$. However, by the assumption that $\dim(\dom(f)) = n$, $\partial f_{\set{F}}$ cannot contain a line, hence $\row(A)$ must intersect a face of $\partial f_{\set{F}}$, which is a contradiction.
\hfill$\square$\end{proof}

\section{Well-posedness of sparse regularized linear regression}\label{sec:well-posedness}

In this section, the well-posedness properties of sparse regularized linear regression problems are studied using the geometric framework for convex piecewise linear functions developed in Section \ref{sec:convex_piecewise_linear_functions}. In Subsection \ref{subsec:regression}, the general class of sparse regularized linear regression and its solution set are introduced, and Subsections \ref{subsec:existence}, \ref{subsec:uniqueness} and \ref{subsec:continuity} discuss the three well-posedness properties of existence, uniqueness and continuity, respectively.

\subsection{sparse regularized linear regression}\label{subsec:regression}

Regularized linear regression problems of the form \eqref{eq:variational_introduction} can be motivated in different ways. For completeness, we mention a popular motivation within statistical inverse problems \cite{kaipio2006statistical} and Bayesian statistics \cite{gribonval2011should}. In a Bayesian setting, both the parameters $\vec{x}$ and the measurements $\vec{b}$ are modeled as random variables. Instead of measuring $\vec{x}$ directly, we measure $A\vec{x}$ which can be corrupted by noise. The relation between the parameters $\vec{x}$ and the corrupted measurements $\vec{b}$ can be described by a likelihood function of the form,
\begin{equation*}
    \pi(\vec{b}\,|\,\vec{x})\ \propto\ \exp\left[-D(A\vec{x},\vec{b})\right].
\end{equation*}
Any a priori knowledge on the parameters $\vec{x}$ can be modeled in a prior distribution as,
\begin{equation*}
    \pi(\vec{x})\ \propto\ \exp\left[-f(\vec{x})\right].
\end{equation*}
Through Bayes' formula, the posterior distribution satisfies $\pi(\vec{x}\,|\,\vec{y})\, \propto\, \pi(\vec{y}\,|\,\vec{x}) \pi(\vec{x})$. There are various ways of obtaining a point estimate from the posterior distribution, e.g., the Minimum Mean Square Error (MMSE) and the Maximum a Posteriori (MAP) estimator. A MAP estimator is defined as a minimizer of the optimization problem,
\begin{equation*}
    \argmax_{\vec{x}} \{\pi(\vec{x}\,|\,\vec{b})\} = \argmin_{\vec{x}} \{-\log(\pi(\vec{b}\,|\,\vec{x})) -\log(\pi(\vec{x}))\} = \argmin_{\vec{x}} \{D(A\vec{x},\vec{b})) + f(\vec{x})\}.
\end{equation*}
Thus, computing a MAP estimate is equivalent to solving optimization problem \eqref{eq:variational_introduction}. 

For the theory discussed in this section, we cannot use arbitrary functions $D$ to measure how well a solution $\vec{x}$ matches the data $\vec{b}$, thus we will often consider the following family of data fidelity functions.
\begin{definition}[Smooth data fidelity]\label{def:smooth_data_fidelity}
We call $D : \reals^m \times \reals^m \rightarrow \reals$ a smooth data fidelity function if the following conditions hold:
\begin{itemize}
    \item The map $\vec{z} \mapsto D(\vec{z}, \vec{b})$ is strictly convex and differentiable for all $\vec{b} \in \reals^m$,
    \item The map $\vec{b} \mapsto \nabla_{\vec{z}} D(\vec{z}, \vec{b}) : \reals^m \rightarrow \reals^m$ is bijective for all $\vec{z} \in \reals^m$,
\end{itemize}
where $\nabla_{\vec{z}} D$ denotes the gradient of $D$ with respect to the first variable.
\end{definition}

Although Definition \ref{def:smooth_data_fidelity} can be quite restrictive, many popular classes of data fidelity functions fall into this class, as stated in the following propositions.

\begin{proposition}\label{prop:bregman_divergence}
Consider the Bregman divergence associated to a continuously differentiable convex function $\phi: \reals^m \rightarrow \reals$, i.e., $D_\phi(\vec{z}, \vec{b}) = \phi(\vec{z}) - \phi(\vec{b}) - \nabla \phi(\vec{b})^T(\vec{z} - \vec{b})$.
Assume that $\phi$ is strictly convex and $\range(\nabla \phi) = \reals^m$, then $D_\phi$ is a smooth data fidelity function.
\end{proposition}
\begin{proof}
The Bregman divergence is always convex in its first argument. The gradient of the Bregman divergence with respect to the first argument is
\begin{equation*}
    \nabla_{\vec{z}}D_\phi(\vec{z}, \vec{b}) = \nabla \phi(\vec{z}) - \nabla \phi(\vec{b}),
\end{equation*}
with, for fixed $\vec{z} \in \reals^m$, the inverse map
\begin{equation*}
    \vec{v} \mapsto [\nabla \phi]^{-1}(\nabla \phi(\vec{z}) - \vec{v}).
\end{equation*}
\hfill$\square$\end{proof}

An example of a Bregman Divergence that satisfies the conditions of Proposition \ref{prop:bregman_divergence} is as follows. Let $\phi(\vec{x}) = \frac{1}{2}\|\vec{x}\|^2_{\Sigma}$ for symmetric positive definite matrix $\Sigma$, then we recover the squared Euclidean distance $D_\phi(\vec{z}, \vec{b}) = \frac{1}{2}\|\vec{z} - \vec{b}\|^2_{\Sigma}$. Another general class of smooth data fidelity functions is the following.

\begin{proposition}\label{prop:smooth_symmetric}
Let $\psi : \reals^m \rightarrow \reals$ be differentiable and strictly convex such that $\nabla \psi : \reals^m \rightarrow \reals^m$ is bijective, then $D(\vec{z}, \vec{b}) = \psi(\vec{z} - \vec{b})$ is a smooth data fidelity function. 
\end{proposition}

An example of a function $\psi$ satisfying the conditions of Proposition \ref{prop:smooth_symmetric} are $\phi(\vec{x}) = \|\vec{x}\|_q^q$ for $q > 1$ resulting in $D(\vec{z}, \vec{b}) = \|\vec{z} - \vec{b}\|_q^q$.

We are interested in studying the well-posedness, i.e., existence, uniqueness and continuity of the solution, of the optimization problem
\begin{equation}\label{eq:regularized_linear_regression}
    \argmin_{\vec{x} \in \reals^n} \left\{D(A\vec{x},\vec{b}) + f(\vec{x})\right\}.
\end{equation}
We therefore define the solution map to this problem as follows.

\begin{definition}[solution map]\label{def:solution_set_map}
Let $D : \reals^m \times \reals^m \rightarrow \reals$ be any function, $A\in \reals^{m\times n}$ be a linear operator and $f:\reals^n \rightarrow \reals \cup \{\infty\}$ be a proper convex lower semi-continuous function. Define the corresponding solution map $\set{S}^D_{A,f} : \reals^{m} \rightarrow \mathcal{P}(\reals^n)$ by
\begin{equation*}
    \set{S}^D_{A,f}(\vec{b}) := \argmin_{\vec{x} \in \reals^n} \left\{D(A\vec{x},\vec{b}) + f(\vec{x})\right\}.
\end{equation*}
\end{definition}

The well-posedness of \eqref{eq:regularized_linear_regression} can be written in terms of the solution map as follows:
\begin{enumerate}
    \item Existence:  $\#\set{S}^D_{A,f}(\vec{b}) \geq 1$ for all $\vec{b} \in \reals^m$,
    \item Uniqueness: $\#\set{S}^D_{A,f}(\vec{b}) \leq 1$ for all $\vec{b} \in \reals^m$,
    \item Continuity: $\vec{b} \mapsto \set{S}^D_{A,f}(\vec{b})$ is a single-valued and continuous map.
\end{enumerate}

Each of these conditions will be discussed separately in the remainder of this section.

\subsection{Existence}\label{subsec:existence}

A general condition for the existence of solutions can be derived that does not depend on the piecewise linear structure of the
regularization function.
\begin{lemma}[range condition for general regularization]\label{lem:range_condition}
Let $f:\reals^n \rightarrow \reals \cup \{\infty\}$ be a proper convex lower semi-continuous function and $D : \reals^m \times \reals^m \rightarrow \reals$ a smooth data fidelity function. Then, 
\begin{equation*}
    \vec{x} \in \range(\set{S}^D_{A,f}) \quad \ntext{if and only if} \quad \row(A) \cap \partial f(\vec{x}) \neq \emptyset.
\end{equation*}
\end{lemma}
\begin{proof}
We have that $\vec{x} \in \set{S}^D_{A, f}(\vec{b})$ if and only if $-A^T \nabla_{\vec{z}} D(A\vec{x}, \vec{b}) \in \partial f(\vec{x})$. Thus, if there exists a $\vec{b} \in \reals^m$ for which $\vec{x} \in \set{S}^D_{A, f}(\vec{b})$, then $-A^T \nabla_{\vec{z}}  D(A\vec{x}, \vec{b}) \in \row(A) \cap \partial f(\vec{x})$. For the converse, let $\vec{z} \in \reals^m$ such that $A^T\vec{z} \in \partial f(\vec{x})$. Choose $b = [\nabla_{\vec{z}}  D(A\vec{x}, \cdot)]^{-1}(-\vec{z})$, then $-A^T \nabla_{\vec{z}}  D(A\vec{x}, \vec{b}) \in \partial f(\vec{x})$.
\hfill$\square$\end{proof}

In the case that the regularization function is convex piecewise linear, then the condition that the row space of $A$ intersects a subdifferential extends to the whole set on which the subdifferential is constant. Thus, Proposition \ref{lem:range_condition} implies the following.
\begin{proposition}[range condition for convex piecewise linear regularization]\label{prop:range_condition_piecewise_linear}
Let $f:\reals^n \rightarrow \reals \cup \{\infty\}$ be a convex piecewise linear function and $D : \reals^m \times \reals^m \rightarrow \reals$ a smooth data fidelity function. For any $\set{F} \in \complex{F}_f$ we have 
\begin{equation*}
    \set{F} \subseteq \range(\set{S}^D_{A,f}) \quad \ntext{if and only if} \quad \row(A) \cap \partial f_{\set{F}} \neq \emptyset.
\end{equation*} 
\end{proposition}
\begin{proof}
By Lemma \ref{lem:range_condition}, $\row(A) \cap \partial f_{\set{F}} \neq \emptyset$ if and only if $\relint(\set{F}) \subseteq \range(\set{S}_{A,f})$. If $\set{H} \in \complex{F}_f$ is a face of $\set{F}$, then by duality also $\row(A) \cap \partial f_{\set{H}} \supseteq \row(A) \cap \partial f_{\set{F}} \neq \emptyset$. Therefore, also $\relint(\set{H}) \subseteq \range(\set{S}_{A,f})$. Because $\set{F}$ is the union of the relative interiors of all of its faces, we obtain $\row(A) \cap \partial f_{\set{F}} \neq \emptyset$ if and only if $F \subseteq \range(\set{S}_{A,f})$.
\hfill$\square$\end{proof}

As a corollary, if part of a set $\set{F} \in \complex{F}_f$ is in the range of the solution map, then the whole set $F$ is in the range of the solution map, resulting in the following.
\begin{corollary}[range of solution set]\label{cor:range_solution_set}
Let $f:\reals^n \rightarrow \reals \cup \{\infty\}$ be a convex piecewise linear function and $D : \reals^m \times \reals^m \rightarrow \reals$ a smooth data fidelity function. Then, there exists a polyhedral subcomplex $\complex{E} \subseteq \complex{F}_f$, such that 
\begin{equation*}
    \range(\set{S}^D_{A, f}) = \bigcup_{\set{F} \in \complex{E}}\set{F}.
\end{equation*}
\end{corollary}

A set $\set{F} \in \complex{F}_f$ for which the condition $\set{F} \subseteq \range(\set{S}^D_{A,f})$ holds is referred to accessible and has been studied in more detail for certain polyhedral norms in \cite{schneider2022geometry,tardivel2021geometry}.

\subsection{Uniqueness}\label{subsec:uniqueness}

Although multiple solutions for fixed $\vec{b}$ can behave very differently, they are still similar in the following way.
\begin{lemma}\label{lem:equal_fit}
Let $f:\reals^n \rightarrow \reals \cup \{\infty\}$ be a proper convex lower semi-continuous function and let $D: \reals^m\times \reals^m \rightarrow \reals$ be strictly convex in its first argument.
For every $\vec{b} \in \reals^m$, if $\vec{x}, \vec{y} \in \set{S}^D_{A,f}(\vec{b})$, then $A\vec{x} = A\vec{y}$ and $f(\vec{x}) = f(\vec{y})$.
\end{lemma}
\begin{proof}
Assume that $\vec{x} \neq \vec{y}$ and $A\vec{x}\neq A\vec{y}$. By strict convexity of $D(\cdot, \vec{b})$ we have that for $\vec{z} = \frac{1}{2}(\vec{x} + \vec{y})$,
\begin{equation*}
    D(\vec{z}, \vec{b}) + f(\vec{z}) < \frac{1}{2}\left[D(A\vec{x}, \vec{b}) + f(\vec{x}) + D(A\vec{y}, \vec{b}) + f(\vec{y})\right] = \min_{\vec{v}\in\reals^n}\left\{D(A\vec{v}, \vec{b}) + f(\vec{v}) \right\}.
\end{equation*}
This contradicts $\vec{x}, \vec{y} \in \set{S}^D_{A, f}(\vec{b})$.

Because $\vec{x}, \vec{y} \in \set{S}^D_{A, f}(\vec{b})$ and $A\vec{x} = A\vec{y}$, we get $f(\vec{x}) = f(\vec{y})$.
\hfill$\square$\end{proof}

Most work for proving necessary and sufficient conditions for uniqueness has already been done in Subsection \ref{subsec:level_sets_linear_subspaces}, specifically in Propositions \ref{prop:finding_two_solutions} and \ref{prop:two_solutions_large_codimension}.

\begin{theorem}(Necessary condition for uniqueness)\label{thm:nec_uniqueness}
Let $D:\reals^m\times\reals^m \rightarrow \reals$ be arbitrary.
If there exist $\vec{x} \in \set{S}_{A, f}^{D}(\vec{b})$ for some $\vec{b} \in \reals^m$ and let $\set{F} \in \complex{F}_f$ containing $\vec{x}$ such that $\codim(\partial f_{\set{F}}) - \Rank(A) = k > 0$, then $\dim(\set{S}_{A, f}^{D}(\vec{b})) \geq k$.

Thus, if $\#[\set{F} \cap \set{S}_{A, f}^{D}(\vec{b})] \leq 1$ for all $\vec{b} \in \reals^m$ and $\set{F} \in \complex{F}_f$, then $\#[\set{F} \cap \set{S}_{A, f}^{D}(\vec{b})] = 1$ must imply $\codim(\partial f_{\set{F}}) \leq \Rank(A).$
\end{theorem}
\begin{proof}
Follows directly from Proposition \ref{prop:finding_two_solutions}.
\hfill$\square$\end{proof}

\begin{theorem}[sufficient condition for uniqueness]\label{thm:suf_uniqueness}
Let $D: \reals^m\times\reals^m \rightarrow \reals$ be strictly convex and differentiable in its first argument and $\dim(\dom(f)) = n$. If $\#\set{S}_{A, f}^{D}(\vec{b}) > 1$ for some $\vec{b} \in \reals^m$, then there exists $\set{F} \in \complex{F}_f$ such that $\row(A) \cap \partial f_{\set{F}}$ and $\codim(\partial f_{\set{F}}) > \Rank(A)$.
\end{theorem}
\begin{proof}
Let $\vec{x}, \vec{y} \in \set{S}_{A, f}^{D}(\vec{b})$ distinct, then by Lemma \ref{lem:equal_fit} $A\vec{x} = A\vec{y}$. Furthermore, by optimality we get
\begin{equation*}
    -A^T\nabla_{\vec{z}}D(A\vec{x}, \vec{b}) \in \partial f(\vec{x}) \quad \text{and} \quad -A^T\nabla_{\vec{z}}D(A\vec{y}, \vec{b}) \in \partial f(\vec{y}).
\end{equation*}
Therefore,
\begin{equation*}
    -A^T\nabla_{\vec{z}}D(A\vec{x}, \vec{b}) \in \partial f(\vec{x}) \quad \text{and} \quad -A^T\nabla_{\vec{z}}D(A\vec{x}, \vec{b}) \in \partial f(\vec{y}),
\end{equation*}
hence $\row(A)\cap \partial f(\vec{x}) \cap \partial f(\vec{y}) \neq \emptyset$. The conclusion follows directly from Proposition \ref{prop:two_solutions_large_codimension}.
\hfill$\square$\end{proof}

Combining Theorems \ref{thm:nec_uniqueness} and \ref{thm:suf_uniqueness} results in the the following necessary and sufficient conditions for uniqueness.
\begin{corollary}[Necessary and sufficient conditions for uniqueness]\label{cor:nec_suf_un}
Let $D : \reals^m \times \reals^m \rightarrow \reals$ be a smooth data fidelity function, $A\in \reals^{m\times n}$ be a linear operator and $f:\reals^n \rightarrow \reals \cup \{\infty\}$ be a convex piecewise linear function. If $\dim(\dom(f)) = n$, then $\#\set{S}_{A,f}^D(\vec{b}) = 1$ for all $\vec{b}\in\reals^m$ if and only if $\row(A) \cap \partial f_{\set{F}} \neq \emptyset$ implies $\codim(\partial f_{\set{F}}) \leq \Rank(A)$. That is, every accessible set $\set{F}$ must satisfy $\codim(\partial f_{\set{F}}) \leq \Rank(A)$.
\end{corollary}

\subsection{Continuity}\label{subsec:continuity}

If the solution map is single valued, then the final condition of well-posedness is continuity. In this work, we focus on the least squares error data fidelity function, to stick with the nature of previous proofs. For more general results on continuity, see \cite{scherzer2009variational}. To prove continuity of the solution map, we first show that the solution map is piecewise linear in the following lemma, followed by gluing together the pieces in a theorem.

\begin{lemma}[uniqueness implies piecewise linear]\label{lem:uniqueness_implies_piecewise_linear}
Let $D : \reals^m \times \reals^m \rightarrow \reals$ be $D(\vec{x}, \vec{y}) = \frac12\|\vec{x} - \vec{y}\|^2_{\Sigma}$, $A \in \reals^{m\times n}$ be a linear operator and $f:\reals^n \rightarrow \reals \cup \{\infty\}$ be a convex piecewise linear function. If $\set{S}^D_{A, f}$ is single-valued, then for all $\set{F} \in \complex{F}_f$, the inverse image $[S^D_{A,f}]^{-1}(\set{F})$ is closed and convex and the restricted solution map $S^D_{A, f}\big|_{[S^D_{A,f}]^{-1}(\set{F})}$ is affine. 
\end{lemma}
\begin{proof}
Let $\vec{\alpha},\vec{\beta} \in [S^D_{A,f}]^{-1}(\set{F})$ such that $\vec{x} = S^D_{A,f}(\vec{\alpha}) \in \set{F}$ and $\vec{y} = S_{A,f}(\vec{\beta}) \in \set{F}$. Then,
\begin{equation*}
    A^T\Sigma(\vec{\alpha} - A\vec{x}) \in \partial f_{\set{F}}\quad \text{and} \quad A^T\Sigma(\vec{\beta} - A\vec{y}) \in \partial f_{\set{F}}.
\end{equation*}
Let $\lambda \in [0, 1]$, then by convexity of $\partial f_{\set{F}}$ we obtain,
\begin{equation*}
    A^T\Sigma([\lambda\vec{\alpha} + (1-\lambda)\vec{\beta}] - A[\lambda\vec{x} + (1-\lambda)\vec{y}])\in \partial f_{\set{F}}.
\end{equation*}
Thus, for all $\alpha,\beta \in [S^D_{A,f}]^{-1}(\set{F})$ and $\lambda \in [0,1]$, we have
\begin{equation*}
    S_{A, f}(\lambda\vec{\alpha} + (1-\lambda)\vec{\beta}) = \lambda S_{A,f}(\vec{\alpha}) + (1-\lambda)S_{A,f}(\vec{\beta}),
\end{equation*}
therefore, $S_{A,f}(\set{F})$ is convex and $[S^D_{A,f}]^{-1}\big|_{[S^D_{A,f}]^{-1}(\set{F})}$ is affine.
Because $[S^D_{A,f}]^{-1}\big|_{[S^D_{A,f}]^{-1}(\set{F})}$ is continuous and $\set{F}$ is closed, the set $[S^D_{A,f}]^{-1}(\set{F})$ is closed.
\hfill$\square$\end{proof}

\begin{theorem}[uniqueness implies continuity]\label{thm:uniqueness_implies_continuity}
Let $D : \reals^m \times \reals^m \rightarrow \reals$ be $D(\vec{x}, \vec{y}) = \|\vec{x} - \vec{y}\|^2_{\Sigma}$, $A \in \reals^{m\times n}$ be a linear operator and $f:\reals^n \rightarrow \reals \cup \{\infty\}$ be a convex piecewise linear function. If $\set{S}_{A, f}$ is single-valued, then $S_{A,f}$ is continuous.
\end{theorem}
\begin{proof}
By Corollary \ref{cor:range_solution_set}, there exists a polyhedral subcomplex $\complex{E} \subseteq \complex{F}_f$ such that $\range(S_{A, f}) = \bigcup_{\set{F}\in \complex{E}}\set{F}$. Due to Lemma \ref{lem:uniqueness_implies_piecewise_linear}, the domain of $S_{A,f}$ can be covered by closed sets, i.e., $\dom(S_{A, f}) = \reals^m = \bigcup_{\set{F}\in \complex{E}}[S^D_{A,f}]^{-1}(\set{F})$. The solution map $S_{A, f}$ restricted to each closed covering element $[S^D_{A,f}]^{-1}(\set{F})$ is affine, hence continuous. Because the function coincided on the intersection of the covering elements, the continuity of $S_{A,f}$ follows from the pasting lemma \cite[Theorem 18.3]{Munkres2014topology}.
\hfill$\square$\end{proof}

Combining the existence and uniqueness result in Proposition \ref{cor:nec_suf_un} with the continuity result in Theorem \ref{thm:uniqueness_implies_continuity} results in slightly more general version of the necessary and sufficient conditions for well-posedness in Theorem \ref{thm:nec_suf_well_posedness_intro}.

\section{Implications}\label{sec:implications}

The conditions derived in the previous section, specifically Corollary \ref{cor:nec_suf_un} for uniqueness and the extension to continuity in Theorem \ref{thm:uniqueness_implies_continuity}, give us more geometric intuition on the well-posedness of sparse regularized linear least squares. In this section, we will use this geometric viewpoint to further study when an optimization problem is well-posed or ill-posed. Particularly, in Subsection \ref{subsec:comparison_l2_l1}, we compare the conditions for well-posedness between regularization of the form $f(\vec{x}) =  \|L\vec{x}\|_2^2$ and $f(\vec{x}) = \|L\vec{x}\|_1$. In Subsection \ref{subsec:strong_ill-posed}, we discuss forward operators $A$ of relatively low rank that, even with specific sparsity promoting regularization, require a large number additional measurements to obtain well-posedness. This leads to a combinatorial look at well-posedness, which will be used in Subsection \ref{subsec:computation_complexity} to study the computational complexity of verifing well-posedness.

\subsection{Comparison between $l_2$ and $l_1$ regularization}\label{subsec:comparison_l2_l1}

Many popular regularization choices make use of the squared $l_2$ norm, resulting in optimization problems of the form,
\begin{equation}\label{eq:tikhonov}
    \argmin_{\vec{x}\in \mathbb{R}^n}\left\{\frac{1}{2}\|A\vec{x}-\vec{b}\|_2^2 + \frac{1}{2}\|L\vec{x}-\vec{c}\|_2^2\right\},
\end{equation}
for some choice of $L \in \mathbb{R}^{k \times n}$. A vector $\vec{x}^\star$ is solution of \eqref{eq:tikhonov} if and only if it satisfies the normal equation $(A^TA + L^TL)\vec{x}^\star = A^T\vec{b} + L^T\vec{c}$, which has a unique solution precisely when $\nullspace(A)\cap\nullspace(L) = \{\vec{0}\}$. Under this condition, the solution map of \eqref{eq:tikhonov} takes the form $\vec{b} \mapsto (A^TA + L^TL)^{-1}(A^T\vec{b} + L^T\vec{c})$, a continuous map. Thus, under the range condition $\nullspace(A)\cap\nullspace(L) = \{\vec{0}\}$, optimization problem \eqref{eq:tikhonov} is well-posed.

Consider each row of $A$ as a separate measurement and each row of $L$ as a separate penalty, such that when we append or remove rows from $A$ and $L$, we can speak of adding or removing measurements and penalties respectively.
Then, from this range condition, we can make the following intuitive observations about problem \eqref{eq:tikhonov}:
\begin{itemize}
    \item Adding measurements can make the problem well-posed, never ill-posed.
    \item Adding penalties can make the problem well-posed, never ill-posed.
    \item Removing measurements can make the problem ill-posed, never well-posed.
    \item Removing penalties can make the problem ill-posed, never well-posed.
\end{itemize}
These observations can be summarized as saying that adding information, either measurements or penalties, can only improve the well-posedness, whilst removing information can only worsen the well-posedness.

To promote sparsity, the smooth penalty $\frac{1}{2}\|\cdot\|_2^2$ in the regularization of \eqref{eq:tikhonov} can be replace by the non-differentiable norm $\|\cdot\|_1$, resulting in an optimization problem of the form,
\begin{equation}\label{eq:l1}
    \argmin_{\vec{x}\in \mathbb{R}^n}\left\{\frac{1}{2}\|A\vec{x}-\vec{b}\|_2^2 + \|L\vec{x}\|_1\right\}.
\end{equation}

From the theory presented in Section \ref{sec:well-posedness}, we know that this problem is well-posed if for any accessible set $\set{F}$, i.e., $\row(A) \cap \partial f_{\set{F}} \neq \emptyset$, it holds that $\codim(\partial f_{\set{F}}) \leq \Rank(A)$. Let us now consider the same four cases as we considered for problem \eqref{eq:tikhonov}:

\begin{itemize}
    \item When adding a measurement to $A$, $\Rank(A)$ can increase by at most $1$, whilst a set $\set{F}$ can become accessible whose codimension becomes arbitrarily large. Thus adding measurements can turn the problem both well-posed and ill-posed.
    \item When adding a penalty to $L$, the dual complex changes, making it possible to turn the problem both well-posed and ill-posed. This is illustrated in Figure \ref{fig:examples_add_l1} by comparing \ref{fig:examples_add_l1_1} to \ref{fig:examples_add_l1_2}.
    \item When removing a measurement from $A$, $\Rank(A)$ can decrease by at most $1$, whilst a problematic accessible set $\set{F}$ that violates $\codim(\partial f_{\set{F}}) \leq \Rank(A) - 1$ can become inaccessible. Thus, removing measurements can turn the problem both well-posed and ill-posed.
    \item Finally, similarly to adding a penalty, removing a penalty can drastically change the dual complex, turning a problem either well-posed or ill-posed. This is illustrated in Figure \ref{fig:examples_add_l1} by comparing \ref{fig:examples_add_l1_2} to \ref{fig:examples_add_l1_1}
\end{itemize}

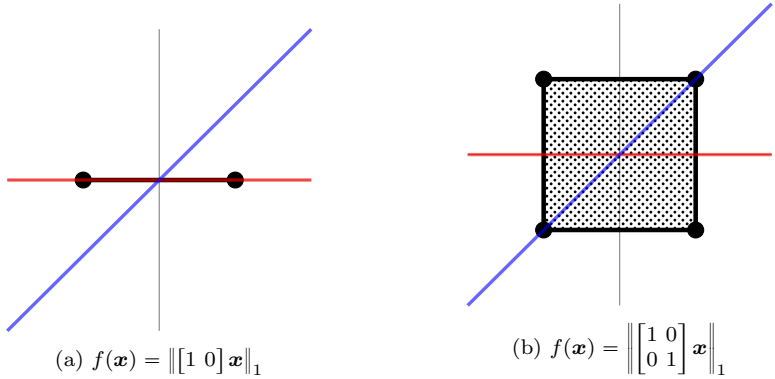
\begin{figure}[H]
\centering
\begin{subfigure}{0.49\textwidth}
    \centering
    \begin{tikzpicture}[scale = 1]
        \draw[gray] (-2,0) -- (2, 0);
        \draw[gray] (0,-2) -- (0, 2);
        
        \draw[black, line width = 0.2em] (-1,0) -- (1, 0);
        
        \filldraw[black] (-1,0) circle (3pt);
        \filldraw[black] (1,0) circle (3pt);
        
        \draw[red, line width = 0.15em, opacity = 0.6] (-2,0) -- (2, 0);
        \draw[blue, line width = 0.15em, opacity = 0.6] (-2,-2) -- (2, 2);
        
    \end{tikzpicture}
    \caption{$f(\vec{x}) = \left\|\begin{bmatrix}1 & 0\end{bmatrix}\vec{x}\right\|_1$}
    \label{fig:examples_add_l1_1}
\end{subfigure}
\hfill
\begin{subfigure}{0.49\textwidth}
\centering
    \begin{tikzpicture}[scale = 1]
        \draw[gray] (-2,0) -- (2, 0);
        \draw[gray] (0,-2) -- (0, 2);
        
        \fill[pattern=crosshatch dots, pattern color=black] (-1,-1) rectangle ++(2,2);
        
        \draw[black, line width = 0.2em] (1,1) -- (1, -1) -- (-1, -1) -- (-1, 1) -- cycle;
        
        \filldraw[black] (1,1) circle (3pt);
        \filldraw[black] (1,-1) circle (3pt);
        \filldraw[black] (-1,-1) circle (3pt);
        \filldraw[black] (-1,1) circle (3pt);
        
        \draw[red, line width = 0.15em, opacity = 0.6] (-2,0) -- (2, 0);
        \draw[blue, line width = 0.15em, opacity = 0.6] (-2,-2) -- (2, 2);
        
    \end{tikzpicture}
    \caption{$f(\vec{x}) = \left\|\begin{bmatrix}1 & 0 \\ 0 & 1\end{bmatrix}\vec{x}\right\|_1$}
    \label{fig:examples_add_l1_2}
\end{subfigure}
\caption{Dual complexes for different $l_1$ regularization functions with linear operators $A = \begin{bmatrix}1 & 0\end{bmatrix}$ ($\row(A)$ is blue) and $B = \begin{bmatrix}1 & 1\end{bmatrix}$ ($\row(B)$ is red).}
\label{fig:examples_add_l1}
\end{figure}

Thus, although the squared $l_2$ norm and $l_1$ norm look very similar, their behaviour with respect to well-posedness is very different. To be more specific, none of the intuitive observations about problem \eqref{eq:tikhonov} hold for problem \eqref{eq:l1}. That is, adding or removing information, either measurements or penalties, can both make or break the well-posedness. As a consequence, one might have to be careful when studying the well-posedness of \eqref{eq:l1}, as the well-posedness or ill-posedness cannot always be directly deduced from slightly larger or smaller problems.

As is discussed in the next subsections, directly verifying well-posedness of linear least squares problems with sparsity promoting regularization can be difficult. However, there are relatively few forward operators $A$ that result in ill-posed problems, in the sense of the following theorem, whose proof is nearly identical to \cite{schneider2022geometry}.
\begin{theorem}\label{thm:almost_surely_wellposed}
    Let $f : \reals^n \rightarrow \cup \{\infty\}$ be a convex piecewise linear function with $\dim(\dom(f)) = n$ and $p$ be the dimension of the smallest subdifferential $\partial f_{\set F} \in \complex{F}_f^{\star}$ that contains the zero vector. If $A \in \reals^{m \times n}$ with $m \geq p$ is a continuously distributed random matrix, then the linear least squares problem for $A$ regularized by $f$ is almost surely well-posed.
\end{theorem}

\subsection{Bad measurements}\label{subsec:strong_ill-posed}
As discussed in the previous subsection, when adding a row to $A$, the rank of $A$ can only increase by at most one, whilst a set $\set{F}$ with arbitrarily large codimension can become accessible. In the worst case scenario, a set satisfying $\codim(\partial f_{\set{F}}) = n$ becomes accessible, in which case the condition for well-posedness becomes $n = \codim(\partial f_{\set{F}}) \leq \Rank(A)$, i.e., $A$ needs to have full column rank. So, even though sparsity can be used to compensate for a lack of measurements, in the sense that $\Rank(A)$ is much smaller than the number of parameters $n$, a single measurement can significantly change the conditions for well-posedness. Specifically, adding a multiple of any vertex from $\complex{F}^\star$ to the rows of $A$ results in needing to add many more measurements before the problem can become well-posed.

As a simple example, consider $f(\vec{x}) = \|\vec{x}\|_1$. If $A \in \mathbb{R}^{m\times n}$ has a single row consisting of $k > 0$ components being $\pm 1$ and all other components being in $(-1,1)$, then this row implies that $k \leq \Rank(A)$ is a necessary condition for well-posedness, independent of the other rows of $A$.

For a more complicated example, consider the following example from computed tomography (CT). Let the parameters $\vec{x} \in \mathbb{R}^{N^2}$ represent a two-dimensional image and consider the measurement matrix $A \in \mathbb{R}^{N\times N^2}$ that computes the sum of each column of the image. In CT, these measurements correspond to a single parallel-beam projection aligned with the vertical axis of the discretization grid. By adding all the rows of $A$ together, we obtain that the all ones vector $\vec{1}$ is in the row space of $A$. However, if we pick $f(\vec{x}) = \gamma\|\vec{x}\|_1$ with $\gamma > 0$ as regularization, then $\{\vec{1}\} \in \complex{F}_f^\star$, hence we require $N^2 \leq \Rank(A)$. Thus, whilst we have only a linear number of measurements, well-posedness requires a quadratic number of measurements. We can therefore conclude that this choice of regularization is ill-suited when we have relatively few of such CT measurements.

Instead, we might pick total variation regularization of the from $f(\vec{x}) = \gamma \text{TV}_{G}(\vec{x})$ with $\gamma > 0$, as discussed in Subsection \ref{subsec:TV}, where the nodes in the graph $G$ are the pixels in the image $\vec{x}$ and the links connect neighbouring pixels. Now let the measurement matrix $A \in \mathbb{R}^{2N\times N^2}$ compute the sum of each column and row of the image separately. In a CT setting, this corresponds to two parallel-beam projections, one aligned with the horizontal axis and one with the vertical axis. We can then show that a necessary condition for well-posedness is $N^2 \leq \Rank(A)$, thus even with TV regularization, we can find realistic measurements that result in requiring $A$ to be full rank for well-posedness.

The three steps to the argument are visualized in Figure \ref{fig:ill-posedness_tv_axes_aligned_measurements}.
First, in Figure \ref{fig:axes_aligned_measurements}, we construct a vector $\vec{z} \in \mathbb{R}^{2N}$ alternating between $-2$ and $2$ with the exception of the rows of $A$ that correspond the boundary of the image, in which case we let it be $-1$ or $1$. The corresponding vector $A^T\vec{z} \in \row(A)$ is shown in Figure \ref{fig:axes_aligned_measurements_transposed}. Due to Proposition \ref{prop:acyclic_vertex}, to show that $A^T\vec{z}$ is a vertex of $\complex{F}_{\text{TV}}^\star$, we can find an directed acyclic graph such that $A^T\vec{z}$ is the difference between the in-degrees and out-degrees of the directed graph. This directed graph is shown in Figure \ref{fig:axes_aligned_measurements_graph}.Recall from Proposition \ref{prop:acyclic_vertex} that adding nonnegativity constraints does not change the vertices and hence this argument also shows that the condition $N^2 \leq \Rank(A)$ holds for nonnegativity constrained total variation.

Similar constructions to these can be done for various different forward operators and choices of regularization, for example, CT with diagonal parallel-beam geometry and total variational regularization. However, many of these are combinatorial in nature and have to be constructed separately. It would be preferred if there was an automated way of performing such type of analysis, which turns out to be computationally difficult, as we discussed in the next subsection. 

\begin{figure}[H]
\centering
\begin{subfigure}{0.3\textwidth}
    \centering
    \begin{tikzpicture}[scale=0.4]
    \draw[step=1.5,black!50,thin,xshift=0.0cm,yshift=0.0cm] (0.0,0.0) grid (6,6);

    \foreach \x in {0,...,3}
        \path [->] (1.5*\x + 0.75, -0.8) edge (1.5*\x + 0.75, 6.8);
    
    \node[] (A) at (1.5*0 + 0.75, -1.25) {\large $-1$};
    \node[] (A) at (1.5*1 + 0.75, -1.25) {\large $2$};
    \node[] (A) at (1.5*2 + 0.75, -1.25) {\large $-2$};

    \foreach \y in {0,...,3}
        \path [->] (-0.8, 1.5*\y + 0.75) edge (6.8, 1.5*\y + 0.75);
    
    \node[] (A) at (-1.3, 1.5*0 + 0.75) {\large $-1$};
    \node[] (A) at (-1.3, 1.5*1 + 0.75) {\large $2$};
    \node[] (A) at (-1.3, 1.5*2 + 0.75) {\large $-2$};
\end{tikzpicture}
    \caption{Measurement matrix $A$ represented by arrows with corresponding numbers representing the residuals $\vec{z}$.}
    \label{fig:axes_aligned_measurements}
\end{subfigure}
\hfill
\begin{subfigure}{0.3\textwidth}
    \centering
    \begin{tikzpicture}[scale=0.45]
    \draw[step=1.5,black!50,thin,xshift=0.0cm,yshift=0.0cm] (0.0,0.0) grid (6,6);

    \node[] (A) at (0.75,0.75)              {\large $-2$};
    \node[] (A) at (0.75 + 1.5,0.75)        {\large $1$};
    \node[] (A) at (0.75 + 3,0.75)          {\large $-3$};
    \node[] (A) at (0.75,0.75 + 1.5)        {\large $1$};
    \node[] (A) at (0.75 + 1.5,0.75 + 1.5)  {\large $4$};
    \node[] (A) at (0.75 + 3,0.75 + 1.5)    {\large $0$};
    \node[] (A) at (0.75,0.75 + 3) {\large $-3$};
    \node[] (A) at (0.75 + 1.5,0.75 + 3) {\large $0$};
    \node[] (A) at (0.75 + 3,0.75 + 3) {\large $-4$};
\end{tikzpicture}
\vspace{2em}
    \caption{Point in the row space $A^T\vec{z}$. \vspace{3.5em}}
    \label{fig:axes_aligned_measurements_transposed}
\end{subfigure}
\hfill
\begin{subfigure}{0.3\textwidth}
    \centering
    \includegraphics[width=1.0\textwidth]{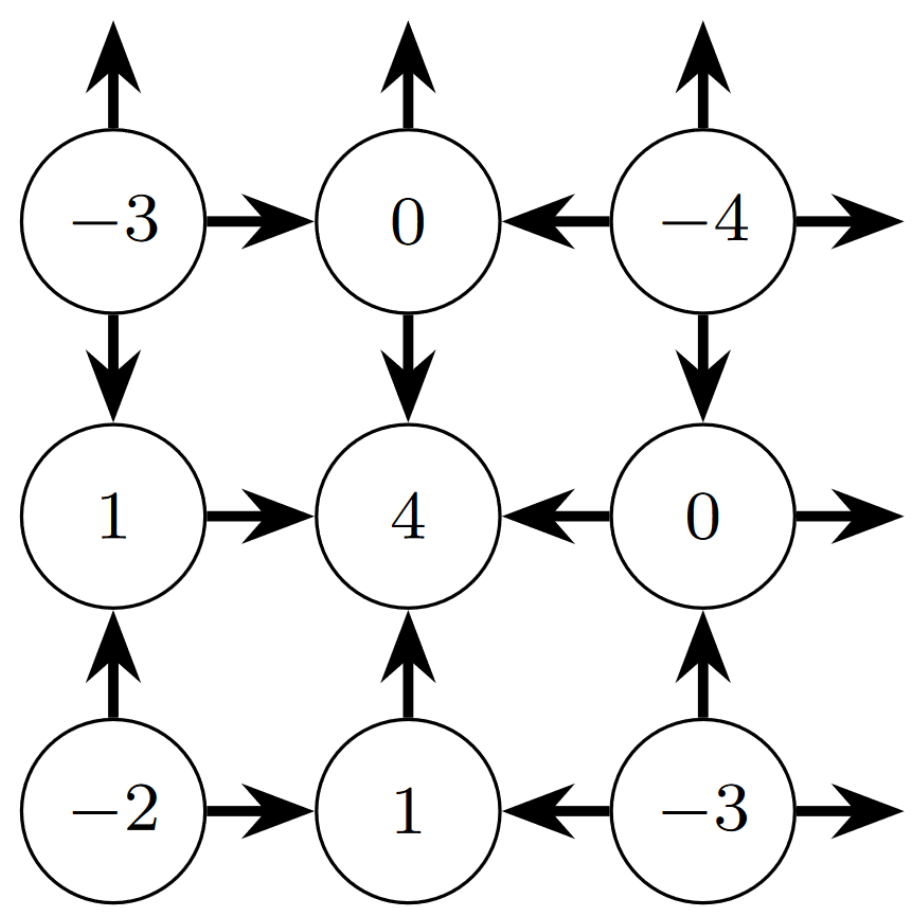}
    \caption{Directed acyclic graph that proofs that $A^T\vec{z}$ is a vertex in the dual complex $\complex{F}_{\text{TV}}^\star$.\vspace{1.3em}}
    \label{fig:axes_aligned_measurements_graph}
\end{subfigure}
\hfill
\caption{Visual proof of the necessary conditions with axes aligned projections and total variation regularization.}
\label{fig:ill-posedness_tv_axes_aligned_measurements}
\end{figure}

\subsection{Computational complexity}\label{subsec:computation_complexity}

As previously shown, a few measurements result in problematic necessary conditions for well-posedness. Specifically, the quantity
\begin{equation}\label{eq:ill-posedness_nr}
    \min\{\dim(\partial f_{\set{F}})\,|\, \partial f_{\set{F}} \in \complex{F}_f^\star \text{ and } \row(A) \cap \partial f_{\set{F}} \neq \emptyset\},
\end{equation}
provides a lower bound on the rank of $A$ that is sufficient to guarantee well-posedness. However, as the following theorem shows, solving \eqref{eq:ill-posedness_nr} for certain $f$ can be transformed into an $l_0$-minimization problem, which is a known to be a computationally difficult problem.

\begin{proposition}\label{prop:l0_minimization}
For any $B\in \reals^{k\times n}$ and $\vec{y} \in \reals^{k}$, there exists $A \in \reals^{m\times n}$ and $f(\vec{x}) = \chi_{\|\cdot\|_\infty \leq 1}(\vec{x}) - \vec{v}^T\vec{x}$ for some $\vec{v} \in \reals^n$, such that the $l_0$ minimization problem
\begin{equation}\label{eq:l0_minimization}
    \min_{\vec{z}}\{\|\vec{z}\|_0 \,|\, B\vec{z} = \vec{y}\},
\end{equation}
is equivalent to
\begin{equation}\label{eq:ill-posedness_minimization}
    \min_{\set{F}}\{\dim(\partial f_{\set{F}})\,|\, \partial f_{\set{F}} \in \complex{F}_f^\star \text{ and } \row(A) \cap \partial f_{\set{F}} \neq \emptyset\}.
\end{equation}

Similarly, the nonnegative $l_0$ minimization problem
\begin{equation}\label{eq:nnl0_minimization}
    \min_{\vec{z}}\{\|\vec{z}\|_0 \,|\, B\vec{z} = \vec{y} \text{ and } \vec{z} \in \reals^{n}_{\geq 0}\},
\end{equation}
can be transformed into \eqref{eq:ill-posedness_minimization} for $f(\vec{x}) = \chi_{\vec{v} - \reals^{n}_{\geq 0}}(\vec{x}) - \vec{v}^T\vec{x}$ for some $\vec{v} \in \reals^n$.
\end{proposition}
\begin{proof}

Let $B\in \reals^{k\times n}$ and $\vec{y} \in \reals^k$. Let $\vec{v} \in \reals^n$ satisfy $B\vec{v} = \vec{y}$ and define $\vec{z} = \vec{x} - \vec{v}$, then \eqref{eq:l0_minimization} is equivalent to
\begin{equation}\label{eq:l0_minimization_translated}
    \min_{\vec{x}}\{\|\vec{x} - \vec{v}\|_0 \,|\, B\vec{x} = \vec{0}\}.
\end{equation}
Now, let $A$ be a matrix such that $\row(A) = \nullspace(B)$, then $B\vec{x} = 0$ is equivalent to $\vec{x} \in \row(A)$.
Furthermore, note that $\|\vec{x} - \vec{v}\|_0 = \dim(\partial f(\vec{y}))$ when $f(\vec{x}) = \chi_{\|\cdot\|_\infty \leq 1}(\vec{x}) - \vec{v}^T\vec{x}$ and $\vec{x} \in \partial f(\vec{y})$. Thus, for these $A$ and $f$, problem \eqref{eq:l0_minimization_translated} is equivalent to,
\begin{equation*}
    \min_{\vec{x}}\{\dim(\partial f(\vec{y}))\,|\, \vec{x} \in \partial f(\vec{y})\text{ and }\vec{x} \in \row(A)\},
\end{equation*}
which is equivalent to \eqref{eq:ill-posedness_minimization}.

The argument for nonnegative $l_0$ minimization is similar.
\hfill$\square$\end{proof}

The (nonnegative) $l_0$ minimization problem is known to be computationally difficult and is used as a motivation for some algorithms in compressed sensing \cite{foucart2013compressive}. Specifically, the difficulty is analyzed using computational complexity theory, for which we now provide a short introduction.

A combinatorial decision problem, i.e., a discrete problem with a yes or no answer, is in the complexity class NP if for every instance of the problem for which the answer is yes, there exists a proof of this answer that can be verified in polynomial computation time. Although the proof can be efficiently verified, it is not known whether for every problem in NP there exists a polynomial time algorithm that can solve the problem. A problem, not necessarily a decision problem, is called NP-hard if a polynomial time algorithm for this problem can be used to create polynomial time algorithms for every problem in NP. Therefore, NP-hard problems are considered at least as difficult as every problem in NP. The (nonnegative) $l_0$ minimization problems are examples of NP-hard problems \cite{natarajan1995sparse,nguyen2019np}, just like many other problems related to sparse estimation \cite{tillmann2013computational}, and therefore the following holds.

\begin{corollary}\label{cor:np_hard_minimization}
Solving
\begin{equation}
    \min\{\dim(\partial f_{\set{F}})\,|\, \partial f_{\set{F}} \in \complex{F}_f^\star \text{ and } \row(A) \cap \partial f_{\set{F}} \neq \emptyset\},
\end{equation}
for convex piecewise linear functions $f$ is NP-hard.
\end{corollary}

Note that Corollary \ref{cor:np_hard_minimization} does not necessarily imply that checking well-posedness itself is NP-hard. Furthermore, the family of regularization functions used in the proof of Proposition \ref{prop:l0_minimization} are not commonly used. Next we will show that verifying ill-posedness is NP-hard for various common regularization functions, even though, due to Theorem \ref{thm:almost_surely_wellposed}, well-posedness is guaranteed almost surely. 

An example of an NP-hard problem that is also in NP, which is called an NP-complete problem, is the partition problem \cite{garey1979computers}. This problem is stated as follows: given a finite multiset of integers $\set{S}$, can $\set{S}$ be partitioned into two sets $\set{S}_1$ and $\set{S}_2$, such that the sum of the elements in $\set{S}_1$ equals the sum of the elements in $\set{S}_2$? To show that verifying ill-posedness is NP-hard, in which case verifying well-posedness is called co-NP-hard, we will show that every instance of the partition problem is an instance of checking ill-posedness.

Therefore, we will restrict ourselves to the matrices satisfying $\Rank(A) = n-1$, such that the corresponding problem is ill-posed precisely if $\row(A)$ contains a 0-dimensional point in $\complex{F}_{f}^{\star}$. For $f(\vec{x}) = \|\vec{x}\|_1$, the 0-dimensional points are precisely the points $\{-1,1\}^n$, i.e., the vertices of the hypercube $[0,1]^n$. Furthermore, $\Rank(A) = n-1$ and thus $\row(A)$ is a hyperplane through the origin, which can be characterized by an non-zero normal vector $\vec{n}$ such that $\nullspace(A) = \text{span}(\vec{n})$. Thus, the optimization problem is ill-posed if and only if $\vec{n}^T\vec{v} = 0$ for some $\vec{v} \in \{-1,1\}^n$. Equivalently, the coefficients in $\vec{n}$ can be partitioned into two sets, such that their respective sums are equal, which is partition problem if we restrict the normal vector to $\mathbb{N}^n$. Thus, we can conclude that verifying the ill-posedness of general $A$ with respect to $f(\vec{x}) = \|\vec{x}\|_1$ is NP-hard, hence verifying well-posedness is co-NP-hard. 

Next, we can extend this argument to $f(\vec{x}) = \text{TV}_G(\vec{x})$ for suitable graphs $G$, or more generally $f_n(\vec{x}) = \|D_n\vec{x}\|_1$ for suitable matrices $D_n$. Using a similar argument as above, ill-posedness is equivalent to $\vec{n}^TD_n^T\vec{v} = 0$ for $\vec{v} \in \{-1,1\}^n$. To transform an instance $\vec{d} \in \mathbb{Z}^m$ of the partition problem to checking ill-posedness with respect to $f_n$, we require a solution of $D_n\vec{n} = \vec{d}$. This can be solved efficiently if $D_n \in \reals^{(n-1)\times n}$ are finite difference matrices or $G$ is a tree graph, hence verifying ill-posedness of total variation is also NP-hard.

As a final extension, consider $f_n(\vec{x}) = \|D_n\vec{x}\|_1 + \chi_{\reals^n_{\geq 0}}(\vec{x})$. For finite difference matrices $D_n \in \reals^{(n-1)\times n}$, the 0-dimensional points of $\complex{F}_{f}^{\star}$ are the same as for $\|D_n\vec{x}\|_1$, thus the co-NP-hardness for verifying well-posedness of anisotropic total variation extends to the nonnegative constrained case. This argument does not extend to $\|\vec{x}\|_1 + \chi_{\reals^n_{\geq 0}}(\vec{x})$, which has only one 0-dimensional point in $\complex{F}^{\star}$ and well-posedness can therefore be checked efficiently if $\Rank(A) = n-1$.

The co-NP-hardness of verifying well-posedness for the three families of regularization functions discussed above is summarised in the following theorem.
\begin{theorem}\label{thm:well-posedness_conp-hard}
Let $f_n(\vec{x})$ be either $\gamma \|\vec{x}\|_1$, $\gamma \|D_n\vec{x}\|_1$ or $\gamma \|D_n\vec{x}\|_1 + \chi_{\reals_{\geq 0}^n}$ with finite difference matrix $D_{n} \in \reals^{(n-1)\times n}$ and $\gamma > 0$. Then, verifying well-posedness of $A\in \reals^{m\times n}$ with respect to $f_n(\vec{x})$ is co-NP-hard.
\end{theorem}

\section{Conclusions}
In this work, we presented a framework, based on convex and polyhedral geometry, for studying the use of convex piecewise linear functions as sparsity promoting regularization for linear regression problems. This framework allowed us to study the well-posedness of the regularized linear regression problem. Particularly, we have provided a geometric characterization of uniqueness for a large family of data fidelity terms, and obtained the same conditions for the well-posedness of regularized linear least squares.

Using these geometric conditions, we highlighted the intuitive differences between using smooth penalties based on the squared $l_2$ norm and those based on the $l_1$ norm, showing that adding or removing measurements and/or regularization terms can turn the $l_1$ norm regularized problem both well-posed and ill-posed. Luckily, well-posedness is almost surely guaranteed when the forward operator is considered random, yet we provided a few combinatorial constructions which can realistically appear in computed tomography that are ill-posed. However, in general, such constructions are computationally difficult to find, as we have shown that for some families of regularization functions, verifying well-posedness is a co-NP-hard problem.

Note that in the definition of well-posedness, we merely required the solution to depend continuously on the data. Often, stronger forms of continuity or stability are preferred, for example Lipschitz continuity, or more concrete quantities like the Lipschitz constant or the condition number of the solution map are favored. Thus, further work could focus on obtaining a deeper understanding of the stability properties of sparsity promoting regularization.

\section{Data availability}
No data has been generated or analyzed for this article.

\section{Acknowledgments}
This work was supported by the Villum Foundation (grant no.\ 25893 and VIL50096) and the Novo Nordisk Foundation (grant no.\ NNF20OC0061894).

\bibliographystyle{spmpsci}
\bibliography{references}

\end{document}